\documentclass[10pt]{amsart}

\usepackage{amsfonts,amssymb}
\usepackage{amsmath}
\usepackage{fancyhdr}
\usepackage[all,cmtip]{xy}

\newtheorem{theorem}{Theorem}
\newtheorem*{theorem*}{Theorem}
\newtheorem{lemma}[theorem]{Lemma}
\newtheorem{corollary}[theorem]{Corollary}

\newtheorem*{corollary*}{Corollary}
\newtheorem{proposition}[theorem]{Proposition}
\newtheorem*{proposition*}{Proposition}

\newtheorem*{claim*}{Claim}

\newtheorem*{thma}{Theorem A}
\newtheorem*{thmaa}{Theorem A$'$}
\newtheorem*{thmb}{Theorem B}
\newtheorem*{thmbb}{Theorem B$'$}
\newtheorem*{facta}{Fact A}
\newtheorem*{factaa}{Fact A$'$}

\theoremstyle{definition}
\newtheorem*{definition}{Definition}
\newtheorem*{definition*}{Definition}

\theoremstyle{remark}
\newtheorem{remark}[theorem]{Remark}
\newtheorem*{remark*}{Remark}

\theoremstyle{plain}

\DeclareMathOperator*{\llimsup}{limsup}
\renewcommand{\limsup}{\llimsup}

\newcommand{\Z}{{\mathbb Z}}
\newcommand{\T}{{\mathbb T}}
\newcommand{\C}{{\mathbb C}}
\newcommand{\D}{{\mathbb D}}
\newcommand{\R}{{\mathbb R}}
\newcommand{\PP}{{\mathbb P}}
\newcommand{\N}{{\mathbb N}}
\newcommand{\Q}{{\mathbb Q}}

\newcommand{\U}{{\mathbb U}}
\newcommand{\CC}{{\mathcal C}}

\newcommand{\CE}{{\mathcal E}}
\newcommand{\CF}{{\mathcal F}}
\newcommand{\CG}{{\mathcal G}}

\newcommand{\CH}{{\mathcal H}}

\newcommand{\CN}{{\mathcal N}}

\newcommand{\CU}{{\mathcal U}}
\newcommand{\CK}{{\mathcal K}}
\newcommand{\CO}{{\mathcal O}}
\begin{document}
\title{Positive Lyapunov Exponents
for Quasiperiodic Szeg\H o Cocycles} \maketitle
\begin{center}
\author{Zhenghe Zhang\\ \footnotesize{Department of Mathematics, Northwestern University}\\
\footnotesize{Email: zhenghe@math.northwestern.edu}}
\end{center}

\fancyhead{}

\fancyhead[CO]{Lyapunov Exponents for Szeg\H o cocycles}

\fancyhead[CE]{Zhenghe Zhang}

\begin{abstract} In this paper we first obtain a formula of averaged Lyapunov
exponents for ergodic Szeg\H o cocycles via the Herman-Avila-Bochi
formula. Then using {\em acceleration}, we construct a class of
analytic quasi-periodic Szeg\H o cocycles with uniformly positive
Lyapunov exponents. Finally, a simple application of the main
theorem in \cite{young} allows us to estimate the Lebesgue measure
of support of the measure associated to certain class of $C^1$
quasiperiodic 2-sided Verblunsky coefficients. Using the same
method, we also recover the \cite{sorets} results for
Schr\"{o}dinger cocycles with nonconstant real analytic potentials
and obtain some nonuniform hyperbolicity results for arbitrarily
fixed Brjuno frequency and for certain $C^1$ potentials.
\end{abstract}

\section{\textbf{Introduction}}
In this paper we study the Lyapunov exponents for two special
families of $SL(2,\C)$ cocycles: Szeg\H o and Schr\"odinger
cocycles. In particular we are interested in how to produce positive
Lyapunov exponents. We first introduce the $SL(2,\C)$ cocycles and
define the associated Lypunov exponents.

\subsection{\textbf{$SL(2,\C)$ cocycles and Lyapunov exponents}}
Let $(X,\mu)$ be a probability space and $T:X\rightarrow X$ be a
$\mu$-preserving transformation. Let $A:X\rightarrow SL(2,\C)$ be a
measurable function satisfying the integrability condition
$$\int_X\ln\|A(x)\|d\mu<\infty.$$ Then we can use $(T, A)$ to define a dynamical
system:
$$(T,A):X\times\C^2\rightarrow X\times\C^2,\ (x,w)\mapsto(T(x),A(x)w).$$
$A$ is the so-called cocycle map. The $n$th iteration of dynamics
will be denoted by $(T, A)^n=(T^n, A_n)$, thus
$$A_n(x)=A(T^{n-1}(x))\cdots A(x), n\geq 1,\ A_0=Id.$$
If furthermore $T$ is invertible, then
$$A_{-n}(x)=A_n(T^{-n}(x))^{-1},\ n\geq 1.$$
One of the most important objects in understanding dynamics of
$SL(2, \C)$ cocycles is the Lyapunov exponent, which is denoted by
$L(T, A)$ and given by
$$\lim\limits_{n\rightarrow\infty}\frac{1}{n}\int_{X}\ln\|A_n(x)\|d\mu
=\inf_n\frac{1}{n}\int_{X}\ln\|A_n(x)\|d\mu\geq 0.$$ The limit
exists and is equal to the infimum since $\{\int_{X}\ln\|A_n(x)\|
d\mu\}_{n\geq1}$ is a subadditive sequence. If in addition $T$ is
$\mu$-ergodic, then by Kingman's subadditive ergodic theorem we also
have
$$L(T, A)=\lim\limits_{n\rightarrow\infty}\frac{1}{n}\ln\|A_n(x)\|,$$
for $\mu$ almost every $x$.

\subsection{\bf{Positive Lyapunov exponents for Schr\"odinger cocycles}}
The Schr\"odinger cocycle map $A^{(E-\lambda v)}:X\rightarrow
SL(2,\R)$ is given by

$$A^{(E-\lambda v)}(x)=
\begin{pmatrix}
E-\lambda v(x)& -1\\
1& 0
\end{pmatrix},$$ where $v:X\rightarrow\R$ is the potential function
(we assume $v\in L^{\infty}(X)$), $E\in\R$ is the energy and
$\lambda$ is the coupling constant. Schr\"odinger cocycles arises
from the one dimensional discrete Schr\"odinger operator $H$ on
$l^2(\Z)$. Let's fix the potential $v$. For $u\in l^{2}(\Z)$, the
Schr\"odinger operator is given by

$$(H_{\lambda,x}u)_n=u_{n+1}+u_{n-1}+\lambda v(T^n(x))u_n.$$

Let $u\in \C^{\Z}$ be a solution of equation $H_{\lambda,x}u=Eu$
(note $u$ is not necessary in $l^2(\Z)$); then the relation between
cocycle and operator is
\[
A^{(E-\lambda v)}(T^n(x)){u_n\choose u_{n-1}}={u_{n+1}\choose u_n}.
\]
Let $\Sigma_{\lambda,x}$ be the spectrum of $H_{\lambda,x}$. That
it,
$$\Sigma_{\lambda,x}=\{E\in\C: H_{\lambda,x}-E \mbox{ is not invertible}\}.$$
Positivity of Lyapunov exponents for Schr\"odinger cocycles are
intensely studied since Lypapunov exponent is very important in
understanding the spectrum of the Schr\"odinger operators. For
potential functions belong to different regularity classes, the
mechanisms lead to positivity of Lyapunov exponents are very
different. We list some of the results which are closely related to
this paper.

\subsubsection{\bf{Continuous potentials}}
We assume $X$ is a compact metric space, $T$ is a homeomorphism
which is also $\mu$-ergodic and $\mu$ is nonatomic. Then
\cite{aviladamanik} (see \cite{aviladamanik} Theorem 1) shows that
there is residual subset $\CG$ of $C(X,\R)$ such that for every
$v\in\CG$, $L(T,A^{(E-v)})>0$ for almost every $E$. What lies behind
this result is actually the monotonicity of the family
$(T,A^{(E-v)})$ with respect to the energy $E$, from which one can
deduce the so-called Kotani Theory (see \cite{kotani2},
\cite{avilakrikorian} or \cite{simon1}). Then one can do the
following steps
\begin{enumerate}
\item $\int^{N}_{-N}L(T,A^{(E-(\cdot))})dE:(L^1(X)\cap
B_{r}(L^{\infty}(X)),\|\cdot\|_1)\rightarrow\R$ is continuous (see
\cite{aviladamanik}, Lemma 1). That is, after some suitable
integration, Lyapunov exponent is a continuous function in $L^1$
sense.
\item the set of simple functions with nonperiodic sequences along orbit
of $T$ is dense in $L^{\infty}(X)$ (see \cite{aviladamanik}, Lemma
2).
\item for $v$ take finitely many values which is
nonperiodic, $L(T,A^{(E-v)})>0$ for almost every $E$ (see
\cite{kotani1}).
\end{enumerate}
(1) and (2) reduces the proof of Theorem 1 in \cite{aviladamanik} to
(3), which is due to the Kotani theory and nondeterminism of
nonperiodic sequences. One can also add a coupling constant since it
doesn't affect the nondeterminism (see \cite{aviladamanik}, Theorem
2). Note positivity of Lyapunov exponents in this case holds only
for a full measure set of energy.

So in this case, it's basically the randomness of potential
functions leads to positivity of Lyapunov exponents.

\subsubsection{\bf{Real analytic potentials}}
To consider higher regularity potentials, we restrict to the case
$(X,\mu,T)=(\R^d/\Z^d,Leb,R_{\alpha})$ and $v\in
C^{r}(\R^d/\Z^d,\R)$, $r\in\Z^+\cup\{\infty,\omega\}$ (here $\Z^+$
is the set of positive integers and $C^{\omega}$ means analyticity),
where $R_{\alpha}:\R^d/\Z^d\rightarrow\R^d/\Z^d$ is the translation
$R_{\alpha}(x)=x+\alpha$. In this case $x$ is the so-called phase
and $\alpha$ is the frequency. For $R_{\alpha}$ ergodic, $v$ is the
so-called quasiperiodic potentials, which is the most intensively
studied.

The first breakthrough, which perhaps is also the most famous one,
is \cite{herman}, where in the case $r=\omega$ and $d=1$ Herman
(among other things) shows that for these $v's$ which are
non-constant trigonometric polynomials, one has $L(\alpha,
A^{(E-\lambda v)})>0$ for all $E$ and all $\alpha$, provided that
$\lambda$ is large, depending only on $v$ (this in fact also holds
for $d>1$). Herman＊s result was later generalized by Sorets and
Spencer \cite{sorets} to all non-constant real-analytic potential
functions. Bourgain \cite{bourgain} generalizes this result to the
case $d>1$ (For higher dimensional {\em Diophantine} frequencies,
it's first proved in \cite{bourgaingoldstein}). For some one
dimensional strong {\em Diophantine} frquencies, Klein generalizes
this result to some Gevrey potentials, see \cite{klein}.

What lies behind these series of results is basically the
analyticity of the potential functions, which implies the
subharmonicity of Lyapunov exponents. This allows one to move the
phase $x$ to complex plane to get around the small divisor problems.
Note here largeness of couplings is needed. But one has that
positivity of Lyapunov exponents holds for all $E$ in these cases,
not just a full measure subset.

Inspired by a new notion, the $accelaration$ of Lyapunov exponents,
which is first introduced in \cite{avila2}, we will give a different
proof of \cite{sorets}'s result, see Theorem A$'$. Our approach can
be applied to more general $SL(2,\C)$ cocycles. The dynamics reasons
which lead to positive Lyapunov exponents is also clearer in our
approach.

In \cite{eliasson}, Eliasson also shows that if $v$ is Gevrey and
satisfies a generic transversality condition (which is also a
generalization of nonconstant real-analytic functions), and if
$\alpha$ is {\em Diophantine}, then for large $\lambda$, the
spectrum is pure point. By the Kotani theory (see \cite{kotani2}),
$L(\alpha, A^{(E-\lambda v)})>0$ for almost every $E$. He also gets
that the measure of the spectrum grows as $\lambda$ goes to
$\infty$.

In these cases, it's basically the analyticity of potential
functions and largeness of couplings lead to positivity of Lyapunov
exponents.

\subsubsection{\bf{Smooth potentials}}
For $r\in\Z^+\cup\{\infty\}$, it's more subtle to produce positive
Lypapunov exponents. Because in these cases, Kotani theory cannot be
easily used to produce positive Lyapunov exponents and there is no
subharmonicity. It seems that a complicated induction and arithmetic
properties on frequencies (like {\em Diophantine} or {\em Brjuno}
conditions) are necessary to take care of the small divisor problems
in these cases.

Early works can be found in \cite{Frohlich} and \cite{sinai}, where
the authors used multi-scale analysis and very special shape of
graph of potential functions is needed. Recently, \cite{bjerklov}
and \cite{chan} obtain some results for some general smooth
potentials. In \cite{bjerklov}, Bjerkl\"ov's approach is close in
spirit to Benedicks-Carleson's approach for H\'enon map (see
\cite{benedickscarleson}). In \cite{chan}, Chan uses multi-scale
analysis; he also obtained positive Lyapunov exponents for all $E$
and most frequencies for some typical $C^3$ potentials via some
variation method. Both of their results need to eliminate
frequencies.

We will prove some similar results with \cite{bjerklov} and
\cite{chan}, see Theorem B$'$. Our proof is a simple application of
a theorem in \cite{young}, but note that the method in \cite{young}
is also close in spirit to Benedicks-Carleson's approach.

The main advantage of our approach is that we can for certain class
of $C^1$ potentials, fix arbitrary {\em Brjuno} frequency to produce
positive Lyapunov exponents. Thus in our approach, it's clearer how
can the geometric properties of potential functions affect the
estimate of Lyapunov exponents. We can also reobtain Eliasson's
result on the estimate of measure of the spectrum for analytic
potentials, see Corollary 7.

In fact, our approach implies that, if one is allowed to eliminate
frequencies, it's enough to assume $v$ is $C^1$ to obtain some
corresponding results (see Remark 14). We also have very precise
description of the `critical sets' for large couplings (see Remark
14).

The other advantage of our approach is again that it's more general.
We will also deal with analytic and smooth cases in an unified form,
in which obstructions to produce positive Lyapunov exponents are
clearer.

For fixed frequency, positivity of Lyapunov exponents for all $E$
with smooth potentials is more subtle. Because one need to deal with
appearance and disappearance of `critical points' in the induction
taking care of the small divisor problems. We are currently working
on it.

\subsection{\bf{Positive Lyapunv exponents for Szeg\H o cocycles}} The Szeg\H o cocycle map
$A^{(E,f)}:X\rightarrow SU(1,1)$ is given by
$$A^{(E,f)}(x)=(1-|f(x)|^2)^{-1/2}\begin{pmatrix}
\sqrt E& \frac{-\overline{f(x)}}{\sqrt E}\\
-f(x)\sqrt E&\frac{1}{\sqrt E}
\end{pmatrix},$$ where $E\in\partial\D$, $\D$ is the open unit disk in
$\C$, and $f:X\rightarrow\D$ is a measurable function satisfying
$$\int_X\ln(1-|f|)d\mu>-\infty.$$
$SU(1, 1)$ is the subgroup of $SL(2,\C)$ preserving the unit disk in
$\C\PP^1=\C\cup\{\infty\}$ under the action by M\"obius
transformations (see Section 2.2 for detailed description). It's
conjugate in $SL(2,\C)$ to $SL(2,\R)$ via
$$Q=\frac{-1}{1+i}
\begin{pmatrix}
1& -i\\
1& i
\end{pmatrix}\in \U(2);$$ that is, $Q^*SU(1,1)Q=SL(2,\R)$. Szeg\H o
cocycles arises naturally in the orthogonal polynomials on unit
circle in the following way. For a polynomial $Q_n$ of degree $n$,
we first define $Q^*_n$, the reversed polynomial, by
$$Q^*_n(E)=E^n\overline{Q_n(1/\bar E)}.$$

Now start with $\varphi_0=\varphi^*_0\equiv1$, we can define a
sequence of polynomials in $E$ as
$$\binom{\varphi_{n+1}}{\varphi^*_{n+1}}=(EA^{(E,f)})(T^n(x))\binom{\varphi_{n}}{\varphi^*_{n}},\ n\geq0.$$
We can use $\varphi_n$ to define a sequence of probability measure
$d\mu_n$ on $\partial\D$ as
$$d\mu_n=\frac{d\theta}{2\pi|\varphi_n(e^{i\theta})|^2}.$$
Then as $n\rightarrow\infty$, $d\mu_n$ converges weakly to a
nontrivial probability measure $d\mu$ on $\partial\D$ (here trivial
means that $\mu$ supported on a finite set). Then
$\{\varphi_n\}_{n\in\N}$ is nothing other than the orthonormal set
of the Hilbert space $\CH=L^2(\partial\D, d\mu)$, which one get by
applying the Gram-Schmidt procedure to $1, E, E^2,\ldots$ In other
words,
\begin{center}$\varphi_n=\frac{P_n[E^n]}{\|P_n[E^n]\|}$,
$P_n$=projection onto $\{1,\ldots,E^{n-1}\}^{\perp}$
\end{center}for $n\in\N$, where $\|\cdot\|$ is the $\CH$ norm.

Here the terms in the sequence $\{f(T^n(x))\}_{n\geq 0}$ appearing
in $A^{(E,f)}(f(T^n(x))$ are called Verblunsky coefficients and
$\mu$ is the associated measure. The correspondence between
$\Pi_0^{\infty}\D$ and the set of nontrivial probability measures on
$\partial\D$ is actually one to one (This is exactly the
Verblunsky's Theorem). One can see \cite{simon2}, Chapter 1 for
detailed descriptions of above discussion.

While Lyapunov exponent is important in understanding the relation
between the Verblunsky coefficients and the associated measure
$\mu$, it's not very intensively studied as in the Schr\"odinger
case. In this paper, inspired by Schr\"odinger case, we will give
some similar results for positivity of Lyapunov exponents for Szeg\H
o cocycles.

First we will show that $A^{(E,f)}$ is actually another typical
monotonic family of $SL(2,\R)$ cocycles, which is of the form
$BR_{\theta}$, where $B$ is a $SL(2,\R)$-valued cocycle map and
$$R_{\theta}=\begin{pmatrix}\cos2\pi\theta&-\sin2\pi\theta\\\sin2\pi\theta&\cos2\pi\theta\end{pmatrix}\in
SO(2,\R).$$

This allows us to apply the Herman-Avila-Bochi formula to obtain a
formula of averaged Lyapunov exponent for erogdic Szeg\H o cocyles,
see Proposition 1. Then as in Schrodinger case, we can easily draw
the conclusion that for a generic set $\CG\subset C(X,\D)$, if
$f\in\CG$, then $L(T,A^{(E,F)})>0$ for almost every $E$. This takes
care of continuous Verblunsky coefficients.

For $r=\omega$, by the method we used to recover \cite{sorets}
result, we can prove for a certain class of analytic quasiperiodic
verblunsky coefficients, $L(T,A^{(E,f)})>0$ for all
$E\in\partial\D$, see Theorem A and Corollary 4. This answers a
question proposed in \cite{simon3}, section 10.16 and
\cite{damanik}, section 3.

For certain class of smooth quasiperiodic Verblunsky coefficients,
we also obtain some results as we do in Schr\"odinger case, see
Theorem B.

We will prove Theorem A and A$'$, Theorem B and B$'$ in unified
ways.

\section{\bf{Statement of main results}}

\subsection{\textbf{A formula of averaged Lyapunov exponent for
erogdic Szeg\H o cocyles}} We first consider the $SL(2,\R)$-valued
cocycle map $A:X\rightarrow SL(2,\R)$ and assume $T$ is
$\mu$-ergodic. In this case, we have the following
Herman-Avila-Bochi formula (\cite{avilabochi}, Theorem 12):
$$\int_{\R/\Z}L(T,AR_{\theta})d\theta=\int_X\ln\frac{\|A(x)\|+\|A(x)\|^{-1}}{2}d\mu.$$
The Herman-Avila-Bochi formula is first obtained as an inequality
(\cite{herman}, $\S6.2$):
$$\int_{\R/\Z}L(T,AR_{\theta})d\theta\geq\int_X\ln\frac{\|A(x)\|+\|A(x)\|^{-1}}{2}d\mu.$$
Since the right hand side is typically positive (unless
$A(\R/\Z)\subset SO(2,\R)$), this gives a lower bound for average
Lyapunov exponent in the family $(T,AR_{\theta})$.

Now in Szeg\H o cocycle family $A^{(E,f)}$, we let $E=e^{2\pi
i\theta}, \theta\in\R/\Z$. Then applying the Herman-Avila-Bochi
formula to the Szeg\H o cocycles, we obtain the following

\begin{proposition}
For $(T,A^{(E,f)})$ as above, we have
$$\int_{\R/\Z}L(T,A^{(E,f)})d\theta=-\frac{1}{2}\int_X\ln(1-|f(x)|^2)d\mu.$$
\end{proposition}
\begin{proof}
We first write $A^{(E,f)}(x)$ as $$(1-|f(x)|^2)^{-1/2}
\begin{pmatrix}
1& -\overline{f(x)}\\
-f(x)&1
\end{pmatrix}\begin{pmatrix}
\sqrt E& 0\\
0&\frac{1}{\sqrt E}
\end{pmatrix}.
$$ Now we conjugate the above matrices to $Q^*A^{(E,f)}Q\in SL(2,\R)$, which is
$$(1-|f(x)|^2)^{-1/2}Q^*
\begin{pmatrix}
1& -\overline{f(x)}\\
-f(x)&1
\end{pmatrix}QR_{-\frac{\theta}{2}}=B(x)R_{-\frac{\theta}{2}}.
$$ It's easily calculated that
$$\|B(x)\|+\|B(x)\|^{-1}=\sqrt{tr(B^{*}B)+2}=\frac{2}{\sqrt{1-|f(x)|^2}},$$
where $tr$ stands for trace. Now by the Herman-Avila-Bochi formula,
we have
$$\int_{\R/\Z}L(T,A^{(E,f)})d\theta=\int_{\R/\Z}L(T,BR_{-\frac{\theta}{2}})d\theta$$
$$=\int_{\R/\Z}L(T,BR_{\theta})d\theta=\int_X\ln\frac{\|B(x)\|+\|B(x)\|^{-1}}{2}d\mu$$
$$=-\frac{1}{2}\int_X\ln(1-|f(x)|^2)d\mu.$$
\end{proof}
Thus if $|f(x)|>0$ on a positive measure set of $x$, we can get
$L(T,A^{(E,f)})>0$ for a positive measure set of $E\in
\partial\D$. As we said in Section 1.3, this computation shows that the typical
monotonic family of cocylces, $(f, AR_{\theta})$ (monotonic with
respect to $\theta$) arise naturally as Szeg\H o cocycles. Then as
in section 1.2.1, we can obtain the following conclusion. Assuming
further that $X$ is a compact metric space, $T$ a homeomorphism and
$\mu$ nonatomic. Then there is a generic set $\CG\in C(X,\D)$ (the
space of continuous functions on $X$ taking value in $\D$) such that
for every $f\in\CG$, we have $L(T,A^{(E,f)})>0$ for $Lebesgue$
almost every $E\in\partial\D.$

As in \cite{aviladamanik}, Remark 4.3, to draw the above conclusion,
we only need to the following replacement in step (1) of Section
1.2.1:
$$v\mapsto\int^{N}_{-N}L(T,A^{(E-v)})dE\mbox{ by }
f\mapsto\int_X\ln(1-|f|^2)d\mu,$$ where the corresponding continuity
conclusion is immediate by the Bounded Convergence Theorem.

As explained in \cite{simon2} (Theorem 12.6.1), an immediate
consequence is that, for every $f\in\CG$ and almost every $x$ and
$Lebesgue$ almost every $\eta\in\partial\D$, the Aleksandrov
measures $d\mu_{x,\eta}$ are pure point. Here the family of
Aleksandrov measures $d\mu_{x,\eta}$ are the measures associated
with the Verblunsky coefficients $\{\eta f(T^n(x))\}_{n\geq0}$,
$\eta\in\partial\D$.

\subsection{\textbf{$C^r$ quasiperiodic Szeg\H o and Schr\"odinger cocycles}}

From now on, we will focus in this paper on the case that $X=\R/\Z$,
$T=R_{\alpha}$ and $\mu$ is the unique probability Haar measure
$dx$. And we will consider only $C^r$ Szeg\H o and Schr\" odinger
cocycles for $r\in\Z^+\cup\{\infty,\omega\}$. We shall use the
notation $(\alpha,A)$ instead of $(T_{\alpha},A)$.

We first introduce the definition of \textit{uniform hyperbolicity},
which plays a central role in this paper. We consider the Riemann
surface $\C\PP^1$, which is given by the projection
$$\pi:(\C^2)^*=(\C^2\setminus\{0\})\rightarrow\C\PP^1=\C\cup\{\infty\},\ \pi(z_0,
z_1)=\frac{z_0}{z_1},\ z_1\neq 0;\ \pi(z_0, 0)=\infty.$$

Then the following commutative diagram induces the M\" obius
transformation for $A=\begin{pmatrix}a&b\\c&d\end{pmatrix}\in
SL(2,\C)$
\[
\xymatrix{
(\C^2)^*\ar[d]_{\pi} \ar[r]^{A} &(\C^2)^*\ar[d]^{\pi}\\
\C\PP^1 \ar[r]^{A\cdot} &\C\PP^1}
\]
In other words, $A\cdot z=\frac{az+b}{cz+d}$. Now we are ready to
give the following definition

\begin{definition*}
We say $(\alpha, A), A\in C^{r}(\R/\Z,SL(2,\C))$ is
\textit{uniformly hyperbolic} if there are two continuous functions
$u,s:\R/\Z\rightarrow\C\PP^1$ such that
\begin{enumerate}
\item $u,s$ are $(\alpha,A)-invariant$, which means that $A(x)\cdot
u(x)=u(x+\alpha)$ and $A(x)\cdot s(x)=s(x+\alpha);$
\item there exists $C>0,\lambda>1$ such that $\|A_{-n}(x)v\|,\|A_n(x)w\|\leq
C\lambda^{-n}$ for every $n\geq1$ and all unit vectors $v\in
u(x),w\in s(x).$
\end{enumerate}
Here $u$ is called the unstable direction and $s$ is the stable
direction of $(\alpha, A)$.

\end{definition*}

We denote the set of uniformly hyperbolic systems by $\CU\CH$. It's
clear that $L(\alpha, A)>0$ for $(\alpha, A)\in\CU\CH$. If
$L(\alpha, A)>0$ and $(\alpha, A)\notin\CU\CH$, we say that
$(\alpha, A)$ is \textit{nonuniformly hyperbolic}. We will denote
the set of nonuniformly hyperbolic systems by $\CN\CU\CH$.

We now state our main results.

\subsubsection{\bf{$C^r$ quasiperiodic Szeg\H o cocyles}}

We've already introduced Szeg\H o cocycles in Section 1. Recall that
the cocycle map is given by
$$A^{(E,f)}(x)=(1-|f(x)|^2)^{-1/2}\begin{pmatrix}
\sqrt E& \frac{-\overline{f(x)}}{\sqrt E}\\
-f(x)\sqrt E&\frac{1}{\sqrt E}
\end{pmatrix}.$$

In Section 1.3 we introduced a way to build the relation between the
1-sided Verblunsky coefficient $\{f(x+n\alpha)\}_{n\geq 0}$ and its
associated measure $d\mu_{\alpha,x}$. Here we need to consider
2-sided Verblunsky coefficients $\{f(x+n\alpha)\}_{n\in\Z}$. We
introduce another way to build this relation in the following. In
fact, there is an unitary operator
$\CC_{\alpha,x}:l^2(\N)\rightarrow l^2(\N)$, the CMV operator,
associated with $\{f(x+n\alpha)\}_{n\geq 0}$. Then $d\mu_{\alpha,x}$
is the spectral measure for $\CC_{\alpha,x}$ and
$\delta_0=(1,0,\ldots,0,\ldots)\in l^2(\N)$. Thus the spectrum of
$\CC_{\alpha,x}$ is the support of the measure $d\mu_{\alpha,x}$.
See \cite{simon2} Theorem 4.2.8 for detailed description.

Now for 2-sided Verblunsky coefficient $\{f(x+n\alpha)\}_{n\in\Z}$,
there is also an unitary operator $\CE_{\alpha,x}:l^2(\Z)\rightarrow
l^2(\Z)$, the induced extended CMV operator associated with it. Let
$\Sigma_{\alpha,x}$ be the spectrum of $\CE_{\alpha,x}$. The
definition of the extended CMV matrix can be found in \cite{simon3},
Section 10.5.34 and 10.5.35. Then we have for irrational frequency,
the following dynamics description of $\Sigma_{\alpha,x}$

\begin{facta}
$\Sigma_{\alpha,x}=\Sigma_{\alpha,0}=\{E: (\alpha, A^{(E,f)})$ is
not uniformly hyperbolic$\}$\footnote{The author is grateful to
David Damanik for pointing this out.}.
\end{facta}
For simplicity, let $\Sigma_{\alpha}=\Sigma_{\alpha,0}$ in this
case. One can see \cite{simon2} and \cite{simon3} for more detailed
discussion.

From now on, we will fix the function $f$ in the Szeg\H o cocycle
families to be of the form $f(x)=\lambda v(x)$ with
$\lambda\in(0,1)$ as the coupling constant, and $v(x)=e^{2\pi
ih(x)}$, where $h\in C^r(\R/\Z,\R/\Z)$. In this case, for irrational
frequency, we denote the spectrum by $\Sigma_{\lambda,\alpha}$.

\begin{remark}
Every function $h\in C^r(\R/\Z,\R/\Z)$ can be written as
$h(x)=kx+\theta(x)$ with $\theta\in C^r(\R/\Z,\R)$ and $k$ is the
degree. We assume $h$ is written in this form.
\end{remark}

Our first theorem is concerned with the analytic case $\theta\in
C^{\omega}(\R/\Z,\R)$. This means that there exists a $\delta>0$
such that $\theta$ can naturally be extended to a holomorphic
function on
$$\Omega_{\delta}=\{z=x+yi\in\C/\Z: |y|<\delta\}.$$

Then $A^{(E,\lambda v)}$ can be extended naturally to a holomorphic
map from $\Omega_{\delta}$ to $SL(2,\C)$ as
$$A^{(E,f)}(z)=(1-\lambda^2)^{-1/2}
\begin{pmatrix}
\sqrt E& -\frac{\lambda}{\sqrt E}\overline{v(\bar z)}\\
-\lambda\sqrt Ev(z)&\frac{1}{\sqrt E}
\end{pmatrix}.$$

We will use $C^{\omega}(\R/\Z,SL(2,\C))$ to denote the set of $A\in
C^{\omega}(\R/\Z,SL(2,\C))$ that can be holomorphically extended to
$\Omega_{\delta}$. We also denote by $C_{\delta}^{\omega}(\R/\Z,\R)$
the set of real analytic functions that can be holomorphically
extended to $\Omega_{\delta}$.

We assume that $\theta\in C_{\delta}^{\omega}(\R/\Z,\R)$ is
nonconstant. Then there is largest positive integer $q=q(\theta)$
such that $\theta(x+\frac{1}{q})=\theta(x)$. Let
$\R/\Z\times\partial\D$ be the set of $(\alpha, E)$ and $\pi_1:
\R/\Z\times\partial\D\rightarrow\R/\Z$ be the projection to the
first component. Then the first main theorem of this paper is

\begin{thma} Let $\theta\in C^{\omega}(\R/\Z,\R)$ be
nonconstant, $k\in\Z$ and $v(x)=e^{2\pi i[kx+\theta(x)]}$. Then
there exists a finite set
$\CF=\CF(\theta,k)\subset\R/\Z\times\partial\D$ with
$\pi_1(\CF)\subset\{\frac{p}{q(\theta)}: p=0,1,\ldots,
q(\theta)-1\}$ with the following property. For any compact set
$\CC\subset(\R/\Z\times\partial{\D})\setminus\CF$, there exists a
constant $c_0=c_0(\theta,k,\CC)\in\R$ such that
$$L(\alpha, A^{(E,\lambda v)})\geq-\frac{1}{2}\ln(1-\lambda)+c_0$$
for all $(\alpha,E;\lambda)\in\CC\times(0,1)$. Moreover, we have
$$L(\alpha, A^{(E,\lambda v)})=0,$$
for all $(\alpha,E;\lambda)\in\CF\times(0,1)$.
\end{thma}

\begin{remark}
Note that we obviously have $L(0,A^{(-1,\lambda v)})=0$ for all $v$,
since it's always the case that $tr(A^{(-1,\lambda v)})=0$. Thus the
finite set $\CF$ always contains $(0,-1)$. The other elements of
$\CF$ are similarly selected using a trace criterion (determined by
$\theta$ and $k$).
\end{remark}
Theorem A easily implies the following corollary:
\begin{corollary}
Let $\theta$, $k$, $v$ and $\CF$ as in Theorem A. Then $\forall
\alpha$ with $d(\alpha,\pi_1(\CF))>0$, there is a constant
$c\in(0,1)$ such that we have that $$L(\alpha, A^{(E,\lambda
v)})>0$$ for all $(E,\lambda)\in\partial\D\times(c,1)$.
\end{corollary}

Thus the uniformly positivity of Lyapunov exponents with respect to
$E$ is established. Before Corollary 4, the only almost periodic
Szeg\H o cocycle with uniformly positive Lyapunov exponents (uniform
in $E$) we know is in \cite{damanik}, where the base dynamics is the
ergodic translation on $\R/\Z\times\Z_2$. Again as in the end of
Section 2.1, an immediate consequence of Corollary 4 is that, for
any irrational $\alpha$ and for $Lebesgue$ almost every $x$ and
$Lebesgue$ almost every $\eta\in\partial\D$, the Aleksandrov
measures $d\mu_{x,\eta}$ (here measures are for 1-sided Verblunsky
coefficients) are pure point for $\lambda$ sufficiently close to 1.

We now turn to the question of estimating the Lebesgue measure of
$\Sigma_{\lambda,\alpha}$ associated to some 2-sided quasiperiodic
Verblunsky coefficients. We will only assume that $\theta$ is $C^1$.
But we have to put some additional conditions on it. Let's first
introduce the notion of {\em Brjuno number}. For $\alpha$
irrational, let $\frac{p_n}{q_n}$ be its $n$th continued fraction
approximant. Then $\alpha$ is called a {\em Brjuno number} if

$$\sum^{\infty}_{n=1}\frac{\ln q_{n+1}}{q_n}<\infty.$$

Note this is a full measure condition since it contains all the {\em
Diophantine numbers}.

We first let $E=e^{2\pi it}$, $t\in\R/\Z$. To make the dependence on
parameters clearer, we write $A^{(t,\lambda v)}$ instead of
$A^{(E,\lambda v)}$. We need two additional assumptions on
$\theta(x)$:
\begin{enumerate}
\item for each irrational $\alpha$, $\theta'(x+\alpha)-\theta'(x)=0$
exactly at two points in $\R/\Z$, which are the unique maximum and
unique minimum of $\theta(x+\alpha)-\theta(x)$ (if we allow $\theta$
to be $C^2$, then an easier but stronger condition is that
$\frac{\partial^2\theta(x)}{\partial x^2}=0$ exactly at two points
in $\R/\Z$);
\item $Leb(\theta(\R/\Z))\leq\frac{1}{2}$.
\end{enumerate}
An example that satisfies these
conditions is $\theta(x)=\frac{1}{2}\cos(2\pi x)$.

Now for any $\epsilon>0$, let
$$\Delta_{\epsilon}(\lambda,\alpha)=\{t: (\alpha, A^{(t,\lambda v)})\in\CN\CU\CH\ and\ L(\alpha, A^{(t,\lambda
v)})>(1-\epsilon)\ln\lambda\}$$ and
$$\Gamma_{\epsilon}(\lambda)=\{(\alpha,t): t\in\Delta_{\epsilon}(\lambda,\alpha)\}.$$
Then our next main theorem is

\begin{thmb}
For any {\em Brjuno} $\alpha$, there is a connected interval
$I_{\alpha}\subset\R/\Z$ of $t$ such that for any $\epsilon>0$,
$$\lim\limits_{\lambda\rightarrow
1}Leb(I_{\alpha}\cap\Delta_{\epsilon}(\lambda,\alpha))=
Leb(I_{\alpha}).$$ Furthermore, let
$\CK=\overline{\cup_{\alpha}I_{\alpha}}$ be the compact domain in
$\R/\Z\times\R/\Z$; then
$$\lim\limits_{\lambda\rightarrow
1}Leb(\CK\cap\Gamma_{\epsilon}(\lambda))=Leb(\CK).$$
\end{thmb}
\begin{raggedright}Now\end{raggedright} by the Fact A, we have $Leb(\Sigma_{\lambda,\alpha})\geq
Leb(I_{\alpha}\cap\Delta_{\epsilon}(\lambda,\alpha)).$ Combined with
Theorem B, we get for Brjuno $\alpha$:
$$\lim\limits_{\lambda\rightarrow 1}Leb(\Sigma_{\lambda,\alpha})=Leb(I_{\alpha}).$$
The equality hold because we will see in the proof of Theorem B that
for every $t\notin I_{\alpha}$, $(\alpha, A^{(t,\lambda
v)})\in\CU\CH$ for $\lambda$ sufficiently close to 1.

\subsection{\textbf{$C^r$ quasiperiodic Schr\"odinger cocycles}}

Now we turn to the Schr\"odinger case. Recall the Schr\"odinger
cocycle map is given by $$A^{(E-\lambda
v)}(x)=\begin{pmatrix}E-\lambda v(x)&-1\\1&0\end{pmatrix}.$$ We
denote the corresponding Schr\"odinger operator by
$H_{\lambda,\alpha,x}$. Let $\Sigma_{\lambda,\alpha,x}$ be the set
of spectrum of the operator $H_{\lambda,\alpha,x}$; then similar to
the Szeg\H o case, for $\alpha$ irrational, the following basic fact
(see \cite{johnson}) is well-known
\begin{factaa}
$\Sigma_{\lambda,\alpha,x}=\Sigma_{\lambda,\alpha,0}=\{E: (\alpha,
A^{(E-\lambda v)})$ is not uniformly hyperbolic$\}$.
\end{factaa}
Let $\Sigma_{\lambda,\alpha}=\Sigma_{\lambda,\alpha,x}$. Again we
are first concerned with the analytic case. Let $v\in
C^{\omega}(\R/\Z,\R)$ be nonconstant. Then we have the following
analogue of Theorem A for Schr\"odinger cocycles:

\begin{thmaa}
Let $v\in C^{\omega}(\R/\Z,\R)$ be nonconstant. Then there exists a
constant $c_0=c_0(v)$ such that
$$L(\alpha,A^{(E-\lambda v)})\geq\ln\lambda+c_0$$
for all $(\alpha,E,\lambda)\in(\R/\Z)\times\R\times(0,\infty)$ .
\end{thmaa}

\begin{remark}
It's interesting to note the difference between Szeg\H o and
Schr\"odinger cases. That in Szeg\H o case we remove a finite set
and consider any compact set in its complement while in
Schr\"odinger case we need to removed nothing. Since the proof of
both cases are basically the same from our method, one can easily
see where this difference comes from.
\end{remark}

Similarly, to estimate the measure of the spectrum of the associated
Schr\"odinger operators, we only need to assume $v$ is $C^1$. But we
need some additional conditions. After a translation and scaling, we
can without loss of generality assume $v(\R/\Z)=[0,1]$. We further
assume that
\begin{enumerate}
\item $v$ has finitely many critical points.
\item for each $t\in[0,1]$ outside of a finite set, $v'(x)$ takes different values for different $x\in v^{-1}(t)$.
\end{enumerate}
Let $t=\frac{E}{\lambda}\in [0,1]$ and $\R/\Z\times[0,1]$ be the set
of $(\alpha,t)$; let $\Delta_{\epsilon}(\lambda,\alpha)$ and
$\Gamma_{\epsilon}(\lambda)$ as in Szeg\H o case. Then our theorem
is

\begin{thmbb}
Fixing arbitrary $\epsilon>0$. Then for each Brjuno $\alpha$,
$$\lim\limits_{\lambda\rightarrow\infty}Leb([0,1]\cap\Delta_{\epsilon}(\lambda,\alpha))=1$$
and
$$\lim\limits_{\lambda\rightarrow\infty}Leb((\R/\Z\times[0,1])\cap\Gamma_{\epsilon}(\lambda))=1.$$
\end{thmbb}

\begin{remark}
Again the corresponding estimate of $Leb(\Sigma_{\lambda,\alpha})$
is immediate by Fact A$'$. Here is different from the Szeg\H o case
since $t$ is related to $E$ in a different way: we know for our $v$,
$\Sigma_{\lambda,\alpha}\subset[-2, \lambda+2]$; thus we get for
Brjuno $\alpha$,
$$\lim\limits_{\lambda\rightarrow\infty}\frac{Leb(\Sigma_{\lambda,\alpha})}{\lambda+4}=1.$$
\end{remark}

It might seem that the conditions we put on $v$ in Theorem B$'$ are
not very natural, but in fact we have the following corollary. First
note that
$$\Sigma_{\lambda,\alpha}\subset[\lambda\inf\limits_{x\in\R/\Z}v(x)-2,
\lambda\sup\limits_{x\in\R/\Z}v(x)+2].$$ Then we also have
\begin{corollary}
Let the frequency $\alpha\in\R/\Z$ be a {\em Brjuno} number and the
potential function $v\in C^{\omega}(\R/\Z,\R)$ be real analytic.
Then we have
$$\lim\limits_{\lambda\rightarrow\infty}\frac{Leb(\Sigma_{\lambda,\alpha})}
{\lambda(\sup v-\inf v)+4}=1.$$
\end{corollary}
\begin{proof}
Again after a translation and scaling we can assume that
$v(\R/\Z)=[0,1]$. Now if $v$ is constant, then by Fact A$'$
$$\Sigma_{\lambda,\alpha}=\{E: |tr(A^{(E-\lambda v)})|\leq2\}=[\lambda v-2,\ \lambda
v+2].$$ Next let's assume that the least positive period of $v$ is
1. Then by analyticity, $\{x\in\R/\Z: v'(x)=0\}$ is a finite set.
This is the condition (1) for $v$ in Theorem B$'$. Let's assume that
$v$ violates condition (2). Then we can assume there are two
sequences $\{x_n\}_{n\geq1}$ and $\{y_n\}_{n\geq1}$ such that
\begin{center}$v(x_n)=v(y_n), v'(x_n)=v'(y_n)\mbox{ and}$\\
$\lim\limits_{n\rightarrow\infty}x_n=x_0\neq
y_0=\lim\limits_{n\rightarrow\infty}y_n.$
\end{center}
Fix a sufficiently small number $\gamma>0$; let $B(x_0,\gamma)$ and
$B(y_0,\gamma)$ be neighborhoods around $x_0$ and $y_0$ with radius
$\gamma$. Then let's consider the analytic curves $(x,v(x))_{x\in
B(x_0,\gamma)}$ and $(x,v(x))_{x\in B(y_0,\gamma)}$. What we want to
show is that these two pieces of analytic curves coincide after a
translation. Thus we are allowed to add some linear function $kx$ so
that we can assume that $v'(x_0)=v'(y_0)\neq0$. Then by the Inverse
Function Theorem, there is a analytic function $f:
v(B(x_0,\gamma))\rightarrow B(x_0,\gamma)$ such that $f[v(x)]=x,x\in
B(x_0,\gamma)$. Now let's consider the analytic function
$g(x)=f[v(x-x_0+y_0)],x\in B(x_0,\gamma)$. By assumption we have
$$g'(x_0-y_0+y_n)=f'[v(y_n)]v'(y_n)=f'[v(x_n)]v'(x_n)=1,n\geq N,$$
where $N$ is some large positive integer. Since
$g(x_0)=f[v(y_0)]=x_0$, we have $g(x)=f[v(x-x_0+y_0)]=x$. By
analyticity, we must have that $v(x-x_0+y_0)=v(x),x\in
B(x_0,\gamma)$ and hence $v(x-x_0+y_0)=v(x),x\in\R/\Z$. Since
$x_0-y_0\in\R/\Z$ is not zero, this contradicts with the assumption
that $v(x)$ is of period 1. Thus Theorem B$'$ implies the corollary
7 for this kind of $v$.

Finally if the least positive period of $v$ is $\frac{1}{n}, n>1$,
then we can instead consider the dynamical system
$(n\alpha,A^{(E-\lambda v(\frac{1}{n}\cdot))})$. Indeed we have the
following facts
\begin{enumerate}
\item $L(\alpha,A^{(E-\lambda v)})=L(n\alpha,A^{(E-\lambda
v(\frac{1}{n}\cdot))})$;
\item $(\alpha,A^{(E-\lambda v)})\in\CU\CH\Leftrightarrow(n\alpha,A^{(E-\lambda
v(\frac{1}{n}\cdot))})\in\CU\CH$;
\item $\alpha$ is {\em Brjuno} $\Rightarrow$ $n\alpha$ is {\em Brjuno}.
\end{enumerate}
These facts obviously reduce the this case to the case that $v$ is
of period 1 and hence prove the corollary.

For the proof of these facts, $(1)$ is straightforward. $(2)$
follows from the fact that $u$ is the unstable direction of
$(\alpha,A^{(E-\lambda v)})$ if and only if $u(\frac{1}{n}\cdot)$ is
the unstable direction $(n\alpha,A^{(E-\lambda
v(\frac{1}{n}\cdot))})$. Same holds for the stable direction $s$.
Finally if $\alpha$ is Brjuno, then for large $s$ the sth continued
fraction approximant $q_s(n\alpha)$ of $n\alpha$ is some
$c_s(n)q_l(\alpha)$, where $q_l(\alpha)$ is the lth continued
fraction approximant of $\alpha$. Then it's not very difficult to
see that $q_{s+1}(n\alpha)=c_{s+1}(n)q_{l+1}(\alpha)$. Here
$\{c_s(n)\}_{s\geq1}$ is a sequence of constants depending only on
$n$ and there is a constant $c>0$ such that
$c^{-1}<c_s(n)<c,s\geq1$. This obviously implies that
$$\sum^{\infty}_{s=1}\frac{\ln q_{s+1}(n\alpha)}{q_{s}(n\alpha)}<\infty,$$ that is
$n\alpha$ is a {\em Brjuno number}. This completes the proof of
Corollary 7.
\end{proof}
The same estimate in Corollary 7 is obtained in \cite{eliasson} (see
\cite{eliasson}, Theorem) for a class of Gevrey potentials and {\em
Diophantine} frequencies, where the class of potentials also
includes all real analytic functions.

\subsection{\textbf{Outline of the remaining part of the paper}}
In Section 3 we introduce $acceleration$ for analytic $SL(2,\C)$
cocycles: we state a main theorem about it and prove a weaker
version which will be enough for this paper. Then we construct a
class of uniformly hyperbolic systems. In Section 4, based on
Section 3, we prove Theorems A and A$'$. In Section 5, we first
state the main theorem in \cite{young} in a slightly different form,
which makes the application easier. Then we use it to prove Theorems
B and B$'$. Finally in Section 6 we
end with some discussion.\\

\begin{raggedright}$\mathbf{Acknowledgements}$.\end{raggedright}
I would like to thank my advisors Artur Avila and Amie Wilkinson for
providing me with the problem of uniform positive Lyapunov exponents
for quasi-periodic Szeg\H o cocycles, sharing with me their ideas
and lots of helpful discussions. I am grateful to Michael Goldstein
and Vadim Kaloshin for some useful discussions and to Fields
Institute for hospitality. I also would like to thank the referees
for many detailed and helpful comments on how to improve the
presentation.

\section{\textbf{Acceleration and uniform hyperbolicity}}
In this section we give a brief introduction of $acceleration$ and
some related results. Then we construct a class of uniformly
hyperbolic systems.
\subsection{\textbf{Quantization of acceleration}}
The main tool we are going to use is the $acceleration$ of
quasiperiodic analytic $SL(2,\C)$ cocycles, which is first
introduced in \cite{avila2}, where lots of important results about
it have been proved. We will only introduce what we need here. If
$A\in C_{\delta}^{\omega}(\R/\Z, SL(2,\C))$, then for each $y\in
(-\delta, \delta)$ we can define $A_y\in C^{\omega}(\R/\Z,
SL(2,\C))$ by $A_y(x)=A(x+iy)$. Then the $acceleration$ is defined
by
$$\omega(\alpha, A)=\lim\limits_{y\rightarrow 0^+}
\frac{1}{2\pi y}(L(\alpha, A_y)-L(\alpha, A)),$$ where the existence
of the limit is guaranteed by the fact that Lyapunov exponent
$L(\alpha, A_y)$ is a convex function of $y$. Indeed, setting
$l(y)=L(\alpha, A_y)$, we have
$$l(y)=\inf_k\frac{1}{2^k}\int_{\R/\Z}\ln\|A_{2^k}(x+iy)\|dx.$$
Thus if we complexify $y$, then $l(y)$ is the limit of a decreasing
sequence of subharmonic functions which implies itself is also
subharmonic; furthermore the change of variable $x\mapsto x+\Im y$
in the integral shows that $l(y)$ is independent of $\Im y$. These
together imply that $l(y)$ is convex in $\Re y$. Note also this fact
holds for all $\alpha$. While we will not use it explicitly, the
following result underlies the philosophy of the proof of Theorems A
and A$'$.

\begin{theorem}[Acceleration is quantized]
The $acceleration$ of analytic $SL(2,\C)$ cocycles with irrational
frequency is always an integer.
\end{theorem}
Note that one immediate consequence of Theorem 8 is that $l(y)$ is
piecewise linear for irrational frequency. The proof of Theorem 8
uses continuity of Lyapunov exponents on
$$((\R/\Z)\setminus\Q)\times\C_{\delta}^{\omega}(\R/\Z, SL(2,\C)),$$
which is proved in \cite{jitomirskaya} (see \cite{jitomirskaya},
Corollary 2). But what we need here is the analogous result in
$\CU\CH$ case, where the frequency can also be rational. As remarked
in \cite{avila2}, the proof in this case is much easier. Let's
formulate and prove it.

\begin{theorem}
The $acceleration$ $\omega(\alpha, A)$ is integer-valued and is
$C^{\infty}$ on
$$\CU\CH_{\omega}=\CU\CH\cap[\R\times C^{\omega}(\R/\Z,SL(2,\C))].$$ Thus $\omega(\alpha,A)$
is constant on any connected component of $\CU\CH_{\omega}$.
\end{theorem}

\begin{proof}
For $(\alpha, A)\in\CU\CH_{\omega}$, we have that the unstable and
stable directions $u,s\in C^{\omega}(\R/\Z,\C\PP^1)$ (see
\cite{avila2}, Lemma 10). Then we can let $B:\R/\Z\rightarrow
SL(2,\C)$ be analytic with column vectors in $u(x)$ and $s(x)$. Then
$B$ diagonalizes A as
$$B(x+\alpha)^{-1}A(x)B(x)=\begin{pmatrix}r(x)& 0\\0&r(x)^{-1}\end{pmatrix}.$$
Then it's easy to see that
$$L(\alpha, A)=\lim\limits_{n\rightarrow\infty}\frac{1}{n}\sum^{n-1}_{j=0}
\int_{\R/\Z}\Re\ln r(x+j\alpha)dx=\int_{\R/\Z}{\Re\ln r(x)}dx.$$
Note that $r:\R/\Z\rightarrow\C$ is also analytic. By openness of
$\CU\CH$ and analyticity, we have that the corresponding upper-left
entry of the diagonalized cocycle for $A_y(x)$ can be chosen to be
$r(x+yi)$. Thus
\begin{align*}
\omega(\alpha,A)&=\lim\limits_{y\rightarrow 0^+}\frac{L(\alpha,
A_y)-L(\alpha, A)}{2\pi y}\\
&=\Re\int_{\R/\Z}\lim\limits_{y\rightarrow 0}\frac{\ln r(x+yi)-\ln
r(x)}{2\pi y} dx\\
&=\Im \frac{1}{2\pi}\int_{\R/\Z}d\ln r^{-1}(x)
\end{align*}
is the winding number of $r^{-1}$ about the origin and hence an
integer.

We also have that $L(\alpha,A)$ is $C^{\infty}$ on $\CU\CH_{\omega}$
(see \cite{avila2}, section 1.3 for a detailed discussion). Thus
$$\omega(\alpha,A)=\frac{1}{2\pi}\frac{\partial}{\partial y}L(\alpha, A_y)$$
is also $C^{\infty}$ on $\CU\CH_{\omega}$.
\end{proof}

It's interesting to note that using Theorem 9 alone, we will give a
self-contained proof of Theorem A and Theorem A$'$ in our paper.

\subsection{\textbf{A class of uniformly hyperbolic systems}}
The main tool we are going to use in this section is the polar
decomposition of $SL(2,\C)$ matrices, which will enable us to
construct a class of uniformly hyperbolic systems. For $A\in SL(2,
\C)$ it's a standard result that we can decompose it as
$A=U_1\sqrt{A^*A}$, where $U_1\in SU(2)$ and $\sqrt{A^*A}$ is a
positive Hermitian matrix. We can further decompose $\sqrt{A^*A}$ as
$\sqrt{A^*A}=U_2\Lambda U_{2}^{*}$, where column vectors of $U_2$
are eigenvectors of $A^*A$ (thus $U_2$ can be chosen such that
$U_2\in SU(2)$) and $\Lambda=diag(\|A\|, \|A\|^{-1})$. Thus
$A=U_1U_2\Lambda U_2^*$. By this decomposition procedure and after
some fixed choice of $U_2$, we can consider $U_1, U_2$ and $\Lambda$
as maps from $SL(2,\C)$ to $SL(2,\C)$ so that for each $A\in
SL(2,\C)$
$$A=U_1(A)U_2(A)\Lambda(A)U_2^*(A).$$ Then we have the following
claim:
\begin{lemma}For some suitable choices of column vectors of
$U_2(A)$, we have
$$U_1,U_2,\Lambda:SL(2,\C)\setminus SU(2)\rightarrow SL(2,\C)$$ are
all $C^{\infty}$ maps. Here $C^{\infty}$ is in the sense that all
these maps are between real manifolds.
\end{lemma}
\begin{proof}
Let $A=\begin{pmatrix}a&b\\c&d\end{pmatrix}\in SL(2,\C)\setminus
SU(2)$. First note that by the above decomposition procedure we know
that $\|A\|^2, \|A\|^{-2}$ are two eigenvalues of $A^*A$, thus
$$tr(A^*A)=|a|^2+|b|^2|+|c|^2+|d|^2=\|A\|^2+\|A\|^{-2}>2\mbox{ and}$$
$$\|A\|^2=\frac{1}{2}(tr(A^*A)+\sqrt{tr(A^*A)^2-4}).$$ Thus
$\|A\|^2$ and hence $\|A\|$ are $C^{\infty}$ on $SL(2,\C)\setminus
SU(2)$. This proves that $$\Lambda\in C^{\infty}(SL(2,\C)\setminus
SU(2),SL(2,\C)).$$ Let $U_2'$ be a matrix that diagonalizes $A^*A$.
Then the column vectors are solutions of the equations
$$(A^*A-\|A\|^2I_2)\binom{x_1}{x_2}=0\mbox{ and }(A^*A-\|A\|^{-2}I_2)\binom{x_1}{x_2}=0,$$
where $I_2$ is the $2\times2$ identity matrix. Thus we can choose
$U_2'$ to be $$\begin{pmatrix}\bar ab+\bar cd,&\bar ab+\bar
cd\\\|A\|^2-|a|^2-|c|^2,&\|A\|^{-2}-|a|^2-|c|^2\end{pmatrix},$$
which has nonzero determinant since $A\notin SU(2)$. Then we can
choose $U_2$ as $$U_2(A)=\frac{1}{\sqrt{\det(U_2'(A))}}U'_2(A).$$
This shows that $$U_2\in C^{\infty}(SL(2,\C)\setminus
SU(2),SL(2,\C)).$$ Finally since
$U_1(A)=AU_2(A)\Lambda^{-1}(A)U_2^*(A)$, we have
$$U_1\in C^{\infty}(SL(2,\C)\setminus
SU(2),SL(2,\C)).$$
\end{proof}

Now we consider $(\alpha, A)$, where $\alpha\in\R/\Z$ and $A\in
C^r(\R/\Z, SL(2,\C))$ with $r\in\N\cup\{\infty,\omega\}$. Let's
assume further that $A(x)$ is $SU(2)$ free. As in Lemma 10, we can
decompose $A(x)$ as $A(x)=U_1(x)U_2(x)\Lambda(x) U_2^*(x)$. Then by
Lemma 10, $U_1(x)$, $U_2(x)$ and $\Lambda(x)$ are $C^r, 0\leq
r\leq\infty$ and $C^{\infty}, r=\omega$ in $x\in\R/\Z$. Let
$U_3(x)=U_1(x-\alpha)U_2(x-\alpha)\in SU(2)$. Then we have
$$U_3(x+\alpha)^*A(x)U_3(x)=\Lambda(x)U(x),$$
where $U(x)=U_2(x)^*U_1(x-\alpha)U_2(x-\alpha)\in SU(2)$. This is
equivalent to that $$(0,U_3)^{-1}(\alpha, A)(0,U_3)=(\alpha,\Lambda
U).$$ Thus we can instead consider the dynamical system
$(\alpha,\Lambda U)$. Note that $\alpha$ has already been involved
in when we transform $(\alpha, A)$ to $(\alpha,\Lambda U)$.

Now let $U(x)=
\begin{pmatrix}
c(x)& -\overline{d(x)}\\
d(x)& \overline{c(x)}
\end{pmatrix}$, so $|c(x)|^2+|d(x)|^2=1$, then we have the following lemma

\begin{lemma}
Let $(\alpha, A)$ be a $SU(2)$ free system with the equivalent
system $(\alpha, \Lambda U)$, where $U$ is as above; assume there
exists a $0<\gamma<1$ such that
$\inf\limits_{x\in\R/\Z}|c(x)|\geq\gamma$; let
$\rho=\frac{1}{\gamma}+\sqrt{\frac{1}{\gamma^2}-1}>1$. If
$$\inf\limits_{x\in\R/\Z}\|A(x)\|=\lambda>\rho,$$ then $(\alpha,
A)\in\CU\CH$. Moreover, we have that
$$L(\alpha,A)\geq\ln\lambda-\ln2\rho$$ for all $\lambda\in(\rho,\infty)$.
\end{lemma}

\begin{proof}
The idea is to consider the projectivized dynamics $(\alpha, \Lambda
U\cdot):\R/\Z\times \C\PP^1\rightarrow \R/\Z\times\C\PP^1$ and
construct an invariant cone field. Here again
$\C\PP^1=\C\cup\{\infty\}$ and $\Lambda U\cdot$ acts on $\C\PP^1$ as
M\" obius transformations, which we introduced at the beginning of
section 2.2.

Let $\lambda_1, \lambda_2$ be arbitrary numbers satisfying
$\lambda>\lambda_2>\lambda_1>\rho$; let $B(\infty, r)=\{z\in\C\PP^1,
|z|>r\}$ and $B(0, r)=\{z\in\C\PP^1, |z|<r\}$ for $r>0$. Then we
have the following facts
\begin{enumerate}
\item $(U(x)\cdot B(\infty,\lambda_1))\cap B(0,\frac{1}{\lambda_1})=\varnothing$ and
\item $\Lambda(x)\cdot \bar B(\infty,\frac{1}{\lambda_1})\subset\Lambda(x)\cdot
B(\infty,\frac{1}{\lambda_2})\subset B(\infty,\lambda)$
\end{enumerate}
for all $x\in\R/\Z$. Given this, we have
\begin{align*}
(\Lambda U)\cdot B(\infty,\lambda_1)&\subset\Lambda\cdot
B(0,\frac{1}{\lambda_1})^c
=\Lambda\cdot\bar B(\infty,\frac{1}{\lambda_1})\\
&\subset\Lambda\cdot B(\infty,\frac{1}{\lambda_2})\subset
B(\infty,\lambda)\subset B(\infty,\lambda_1).
\end{align*}

Namely, $B(\infty,\lambda_1)$ is an invariant cone field for
$(\alpha, \Lambda U)$ which is also uniformly contracted into a
sub-disk. Then $(\alpha, \Lambda U)$ and $(\alpha, A)$ is
\textit{uniformly hyperbolic}. Here we use the following two facts
\begin{enumerate}
\item $\CU\CH$ is conjugate invariant. Indeed, if $(\alpha,
A)\in\CU\CH$ and $u,s$ are the two associated invariant sections,
then for arbitrary $B\in C^0(\R/\Z,SL(2,\C))$, $B(x)^{-1}\cdot u(x)$
and $B(x)^{-1}\cdot s(x)$ are the two invariant sections of
$$B(x+\alpha)^{-1}A(x)B(x);$$
\item $(\alpha, A)\in\CU\CH$ is equivalent to the existence of
invariant conefield (see \cite{avila1}, section 2.1).
\end{enumerate}

For the proof of $(1)$ and $(2)$, $(2)$ is obvious. For $(1)$, just
note that when $|c|\geq\gamma$,
$|d|=\sqrt{1-|c|^2}\leq\sqrt{1-\gamma^2}$ and let
$t=\frac{|c|}{|d|}\geq\frac{\gamma}{\sqrt{1-\gamma^2}}$, then
$$|U\cdot(re^{i\theta})|\geq\frac{tr-1}{t+r}>\frac{1}{r}$$
for all $r>\rho$.

In fact, if we consider the function $g(t,r)=\frac{tr-1}{t+r},
t,r>0$, it's symmetrical in $t$ and $r$, and increasing in both
variables. And for $t=\frac{\gamma}{\sqrt{1-\gamma^2}}$,
$g(t,r)=\frac{1}{r}$ exactly at $r=\rho$.

For the estimate of the Lyapunov exponent, note that for $z\in
B(\infty,\lambda_1)$, $|U\cdot z|\geq\frac{1}{\lambda_1}$ by the
above argument. Thus for arbitrarily fixed
$\lambda_1\in(\rho,\lambda)$,
$$\frac{\|\Lambda U(z,1)^t\|}{\sqrt{|z|^2+1}}=\frac{\|\Lambda U(z,1)^t\|}{\|U(z,1)^t\|}\geq
\frac{\|\Lambda
(\frac{1}{\lambda_1},1)^t\|}{\sqrt{\frac{1}{\lambda_{1}^2}+1}}
\geq\sqrt{\frac{\lambda^{-2}+(\frac{\lambda}{\lambda_1})^2}{1+\lambda_{1}^{-2}}}\geq\frac{\lambda}{2\lambda_1}.$$
Thus we have that for $z\in B(\infty,\lambda_1)$
\begin{align*}
L(\alpha,A)&\geq\lim\limits_{n\rightarrow\infty}\frac{1}{n}\ln\|(\Lambda U)_n(x)\frac{(z,1)^T}{\sqrt{|z|^2+1}}\|\\
&\geq\ln\lambda-\ln2\lambda_1,
\end{align*}
where the first inequality follows from the definition of Lyapunov
exponents and invariance of $B(\infty,\lambda_1)$. Now since the
above inequality holds for all $\lambda_1>\rho$, we get
$$L(\alpha,A)\geq\ln\lambda-\ln2\rho.$$
\end{proof}

We continue to use the notation $A(x), U(x), U_1(x), U_2(x),
\Lambda(x), c(x), d(x)$ in the remaining part of this paper.

\section{\textbf{Uniformly positive Lyapunov exponents}}
In this section we prove Theorem A and Theorem A$'$ in an unified
way. Recall that the family of Szeg\H o cocycles are given by
$(\alpha,A^{(e^{2\pi it},\lambda v)}), 0<\lambda<1$, where
$$A^{(E,\lambda v)}(z)=(1-\lambda^2)^{-1/2}
\begin{pmatrix}
\sqrt E& -\frac{\lambda}{\sqrt E}\overline{v(\bar z)}\\
-\lambda\sqrt Ev(z)&\frac{1}{\sqrt E}
\end{pmatrix};$$ the family of Schr\"odinger cocycles are given by
$(\alpha,A^{(\lambda(t-v))}),0<\lambda<\infty$, where
$$A^{(\lambda(t-v)}(z)=
\begin{pmatrix}
\lambda(t-v(z))& -1\\
1& 0
\end{pmatrix}.$$ But for Schr\"odinger cocycles, we will use a slight different form (see Section 4.2).
In both cases we have $z\in\Omega_{\delta}$, where
$$\Omega_{\delta}=\{z=x+yi\in\C/\Z: |y|<\delta\}$$ is the strip to
where the analytic cocycle maps can be extended. Thus they share the
parameters $(z;\alpha,t;\lambda)$. For simplicity, let's write the
cocycle maps as $A^{(t,\lambda)}$ in both cases.  Then the unified
strategy is
\begin{enumerate}
\item As in Lemma 10, we decompose $(\alpha,A)$ as $(\alpha,\Lambda U)$. Then we show that
$\|A^{(t,\lambda)}(z)\|$ are of size $\sqrt{\frac{1}{1-\lambda}}$ in
Szeg\H o case and are of size $\lambda$ in Schrodinger case, where
the largeness is independent of $t$ and $z$. Thus $A^{(t,\lambda)}$
are $SU(2)$ free for $\lambda$ close to 1 in Szeg\H o case and for
$\lambda$ sufficiently large in Schr\"odinger case. We also compute
the upper left entry $c(z;\alpha,t;\lambda)$ of $U$ explicitly.
\item Reduce the uniform positivity (uniform in $z$) of $|c(z;\alpha,t;\lambda)|$ to that
of some holomorphic function $|g(z;\alpha,t)|$ for $\lambda=1$ in
Szeg\H o case and for $\lambda=\infty$ in Schr\"odinger case. Then
we can fix in each case a compact region $\CK\subset\R/\Z\times I$
of $(\alpha,t)$, $I\subset\R$ is a compact interval, with the
following property. For each $(\alpha,t)\in\CK$, the algebraic set
$\{z\in\Omega_{\delta}:g(z;\alpha,t)=0\}$ is finite.
\item Then by the fact
$$\lim\limits_{\lambda\rightarrow\lambda_0}\|c(x+iy;\alpha,t;\lambda)-
c(x+iy;\alpha,t;\lambda_0)\|_{x\in\R/\Z}=0,$$ where $\lambda_0=1$ in
Szeg\H o case and $\lambda=\infty$ in Schr\"odinger case, we show
that for each $(\alpha_j,t_j)\in\CK$, there is some small connected
open set $\CO_{j}\subset\R/\Z\times I$ containing $(\alpha_j,t_j)$,
some $\lambda_j>0$ and height $y_j$ such that
$|c(x+iy_j;\alpha,t;\lambda)|$ is uniform bounded away from zero for
all
$(x;\alpha,t;\lambda)\in\R/\Z\times\CO_j\times[\lambda_j,\infty)$.
\item Using Lemma 11 to get uniform hyperbolicity and estimate
the Lyapunov Exponents.
\item Using compactness of $\CK$ to get finitely many $j$, say $j=1,\cdots,l$,
such that $\CK\subset\bigcup_j\CO_j$. Using Theorem 9 to find an
unique {\em acceleration} $n_j$ on each
$\CO_j\times[\lambda_j,\infty)$.
\item Passing the estimate of Lyapunov exponents from $y_j$ to $y=0$
for each $(\alpha,t)\in\CO_j$ via maximal {\em acceleration}
$n=\max\{n_1,\cdots,n_l\}$. Then we get uniformly positive Lyapunov
exponents on $\CK$.
\end{enumerate}
By this strategy, it's clear that how we can actually construct a
certain class of parametrized analytic $SL(2,\C)$ cocycles with
uniformly positive Lyapunov exponents.
\subsection{\textbf{Proof of main theorem: Theorem A}}

Now we are ready to prove the Theorem A. Let's do it step by step
as introduced in the beginning of this section.\\

\begin{raggedright}\textit{Step} 1.\end{raggedright}
We start with the polar decomposition of $A=A^{(E,\lambda v)},
E=e^{2\pi it}$. For simplicity, let $z=x+yi$ and $v(z)=r(z)e^{2\pi
ih(z)}=re^{2\pi ih}$, where both $r$ and $h$ are real valued
function. Let $a=a(\lambda,
r)=\lambda(r-r^{-1})+\sqrt{4+[\lambda(r-r^{-1})]^2}$, then obviously
both $r$ and $a$ are uniformly bounded away from $\infty$ and $0$ in
any compact subregion of $\Omega_{\delta}$. A direct computation
shows
$$\|A(z)\|=\sqrt{\frac{2+\lambda^2(r^2+r^{-2})+\lambda(r+r^{-1})
\sqrt{4+[\lambda(r-r^{-1})]^2}}{2(1-\lambda^2)}},$$ which is
uniformly of size $\sqrt{\frac{1}{1-\lambda}}$. In particular, if
$y=0$, then $r=1$ and $\|A(x)\|=\sqrt{\frac{1+\lambda}{1-\lambda}}$
for all $x\in\R/\Z$. Also we get
$$U_2(z)=\frac{1}{\sqrt{a^2+4}}\left(
\begin{array}{cc}
a& \frac{2}{E}e^{-2\pi ih}\\
-2Ee^{2\pi ih}& a
\end{array}
\right).$$ Since it's easy to see that
$$U(z)=U_2(z)^*A(z-\alpha)U_2(z-\alpha)\Lambda(z-\alpha)^{-1},$$we obtain the upper-left coefficient of $U$ is
\begin{center}
$c(z;\alpha,E;\lambda)=\frac{c_1}{\sqrt
E}\{a(z-\alpha)a(z)E+4e^{2\pi
i[h(z-\alpha)-h(z)]}+$\\$\frac{2E\lambda a(z)}{r(z-\alpha)}+2\lambda
r(z-\alpha)a(z-\alpha)e^{2\pi i[h(z-\alpha)-h(z)]}\},$
\end{center} where
$$c_1=\frac{\|A(z-\alpha)\|^{-1}}{\sqrt{(a(z)^2+4)(a(z-\alpha)^2+4)(1-\lambda^2)}}$$
is uniformly bounded away from $\infty$ and $0$ for all $z$ in any
compact subregion of $\Omega_{\delta}$ and all $\lambda\in[0,1]$.\\

\begin{raggedright}\textit{Step} 2.\end{raggedright}
For $\lambda=1$, we have $a(1,r)=2r$, and thus
\begin{align*}c(z;\alpha,E;\lambda)&=4\frac{c_1}{\sqrt E} e^{-2\pi
ih(z)}[r(z-\alpha)+r^{-1}(z-\alpha)][Ev(z)+v(z-\alpha)]\\
&=c_2[Ev(z)+v(z-\alpha)].
\end{align*} where again $|c_2|, |c^{-1}_2|$ are uniformly bounded. Thus we
can reduce the analysis of uniform positivity of $|c(z;\alpha,E;1)|$
to that of $|g(z;\alpha, t)|$ with
$g(z;\alpha,E)=Ev(z)+v(z-\alpha)$. Then we get
$$Ev(z)+v(z-\alpha)=0\Leftrightarrow\theta(z)
-\theta(z-\alpha)=\frac{1}{2}-t-k\alpha+m,$$ where $m$ is any
integer; since we can choose $\delta$ such that $\theta(z)$ is
uniformly bounded on $\Omega_{\delta}$, we only need to care about
finitely many such $m$. Now since $q$ is the largest positive
integer such that $\theta(z+\frac{1}{q})=\theta(z)$ and $\theta$ is
a nonconstant real analytic function, these together imply that

\begin{enumerate}
\item $\theta(z)-\theta(z-\alpha)=\frac{1}{2}-t-k\alpha+m,\ \forall
z\in\Omega_{\delta}$, only if $(\alpha,
t)=(\frac{p}{q},\frac{1}{2}+m-k\frac{p}{q})$, where $p=0,1,\ldots,
q-1$. Obviously, such pair $(\alpha,t)$ is finite, where
$(\alpha,e^{2\pi it})$ is nothing other than our $\CF$ in the main
theorem. Let's also denote $\CF$ as such pairs in $(\alpha,t)$.
\item $\theta(z)-\theta(z-\alpha)=\frac{1}{2}-t-k\alpha+m$, has at most
finitely solutions in $\Omega_{\delta}$ otherwise.\\
\end{enumerate}

\begin{raggedright}\textit{Step} 3.\end{raggedright}
Now let $\CC$ be as in the Theorem A. Then each $(\alpha_j,e^{2\pi
it_j})\in\CC$ satisfying condition (2) above. So for any
$(\alpha_j,e^{2\pi it_j})\in\CC$, we can find some height $y_j$ such
that
$$|c(x+iy_j;\alpha_j,t_j;1)|$$
is bounded away from zero for all $x\in\R/\Z$. Then for each
$(\alpha_j,e^{2\pi it_j})\in\CC$, we can find some connected open
set $\CO_j$ satisfying $(\alpha_j,e^{2\pi
it_j})\in\CO_j\subset(\R/\Z\times\partial\D)\setminus\CF$ and some
large $\lambda_j>0$ such that
$$|c(x+iy_j;\alpha,t;\lambda)|$$
is bounded away from zero for all $(x,\alpha,e^{2\pi
it},\lambda)\in\R/\Z\times\CO_j\times[\lambda_j,1]$. Here we use the
straightforward fact that for fixed $(y_j,\alpha_j,t_j)$, as
$(\alpha,t,\lambda)\rightarrow(\alpha_j,t_j,1)$
$c(x+iy_j;\alpha,t;\lambda)\rightarrow c(x+iy_j;\alpha_j,t_j;1)$ in
$C^0(\R/\Z,\R)$ as function of $x$. On the other hand, we know that
$\|A(x)\|$ is uniformly large and of size
$\sqrt{\frac{1}{1-\lambda}}$.\\

\begin{raggedright}\textit{Step} 4.\end{raggedright}
Now by Lemma 11, without loss of generality, we can assume
$\lambda_j$ is such that there is a constant
$\eta_j=\eta_j(\lambda_j,\CO_j)$ and for each
$(\alpha,E,\lambda)\in\CO_j\times[\lambda_j,1)$, the following two
things hold
\begin{enumerate}
\item $(\alpha, A_{y_j}^{(E,\lambda v)})\in\CU\CH$;
\item $L(\alpha,A_{y_j}^{(E,\lambda v)})\geq-\frac{1}{2}\ln(1-\lambda)+\eta_j$.\\
\end{enumerate}

\begin{raggedright}\textit{Step} 5.\end{raggedright}
Now by compactness of $\CC$, there exist finitely many $j$, say
$j=1,2,\ldots, l$, such that $\CC\subset\bigcup_{1\leq j\leq
l}\CO_j$. Let $\lambda_0=\max\{\lambda_1,\ldots,\lambda_l\}$; by
Theorem 9, there are at most $l$ integers, say $n_j\geq 0$, such
that $\omega(\alpha,A_{y_j}^{(E,\lambda v)})=n_j$ on
$\CO_i\times[\lambda_0,1)$. Here $n_j\geq 0$ follows from the fact
that $L(\alpha,A_y)$ is convex in $y$ and $A^{(E,\lambda v)}\in
SU(1,1)$ when $y=0$. Let $n_0=\max\{n_1,\ldots, n_l\}$ and
$y_0=\max\{y_1,\ldots,y_l\}$.\\

\begin{raggedright}\textit{Step} 6.\end{raggedright}
Now by the definition of \textit{acceleration}, we have the
following estimates: for all $(\alpha,E,\lambda)\in
\CC\times[\lambda_0,1)$, there exists $j$, $1\leq j\leq l$ such that
$$L(\alpha,A^{(E,\lambda v)})\geq L(\alpha,A_{y_j}^{(E,\lambda
v)})-\frac{n_0y_0}{2\pi}\geq-\frac{1}{2}\ln(1-\lambda)+\eta_j-\frac{n_0y_0}{2\pi}.$$
Thus for all $(\alpha,E,\lambda)\in \CC\times[\lambda_0,1)$,
$$L(\alpha,A^{(E,\lambda v)})\geq
-\frac{1}{2}\ln(1-\lambda)+c_0$$ with $c_0=\min\limits_{1\leq j\leq
l}\{\eta_j-\frac{n_0y_0}{2\pi}\}$. Obviously we can change $c_0$ to
let that $$L(\alpha,A^{(E,\lambda v)})\geq
-\frac{1}{2}\ln(1-\lambda)+c_0$$ for all $(\alpha,E,\lambda)\in \CC\times(0,1)$.
This completes the proof of first part of Theorem A.\\

Finally, for $(\alpha, E)\in\CF$, we can compute it directly. In
fact, we've already computed above that for $y=0$,
$\|A(x)\|=\sqrt{\frac{1+\lambda}{1-\lambda}}$. We also have $r=1$
and $a=2$. Thus the above formula for $c(x)$ shows
\begin{align*}
c(x;\alpha,E;\lambda)&=4\frac{c_1}{\sqrt E} e^{-2\pi
ih(x)}(\lambda+1)[Ev(x)+v(x-\alpha)]\\
&=\frac{1}{2\sqrt E}e^{-2\pi
ih(x)}[Ev(x)+v(x-\alpha)]\\
&=c_3[Ev(x)+v(x-\alpha)],
\end{align*}
where $|c_3|, |c_{3}^{-1}|$ is again uniformly bounded.

Now we see to solve the equation $c(x)=0$, one goes back to the case
$\lambda=1$; only here $y$ is fixed to be $0$. First let's note in
the rational case the Lyapunov exponents can be expressed as

$$L(t/s,A)=\frac{1}{s}\int_{\R/\Z}\ln\delta(A_{t/s}(x))dx,$$
where $A_{t/s}(x)=A(x+(s-1)t/s)\cdots A(x)$ and $\delta(B)$ is the
spectral radius of $B$. Namely,
$$\delta(B)=\lim\limits_{n\rightarrow\infty}\|B^n\|^{\frac{1}{n}}
=\inf\limits_{n\geq1}\|B^n\|^{\frac{1}{n}}.$$

Now for $(\alpha,E)\in\CF$, $c(x)=0$ for all $x$ and $\lambda$,
since $\|A(x)\|=\sqrt{\frac{1+\lambda}{1-\lambda}}$ is constant,
it's easy to see $tr[(U\Lambda)_{p/q}(x)]\in\{-2,0,2\}$ for all $x$.
This obviously implies $\delta[(U\Lambda)_{p/q}(x)]=1$. Thus,
$\delta(A_{p/q}(x))=1$ for all $x$, which implies
$$L(\alpha,A^{(E,\lambda v)})=0$$
for all $(\alpha,E,\lambda)\in\CF\times(0,1)$. This completes the
proof of main theorem.

\subsection{\textbf{Recovery of the Schr\"odinger case: proof of Theorem A$'$}}

Without loss of generality, we assume $|v(z)|\leq 1$ on
$\Omega_{\delta}$. Before we do our common steps with section 4.1,
let's first use a simple trick to avoid that $A^{(E-\lambda v)}$ can
always touch $SU(2)$ for $t=\frac{E}{\lambda}\in v(\R/\Z)$, which
leads to the discontinuity of the polar decomposition (this simple
trick is also crucial in the proof of Theorem B$'$, where uniform
largeness of $\|A^{(E-\lambda v)}\|$ is required). We instead
consider $\hat A=TA^{(E-\lambda v))}T^{-1}$, where
$$T=\left(
\begin{array}{cc}
\sqrt{\lambda}^{-1}& 0\\
0& \sqrt{\lambda}
\end{array}
\right).$$ This obviously doesn't change the dynamics. Now let's
carry out our steps.\\

\begin{raggedright}\textit{Step} 1.\end{raggedright}
Now we instead consider the polar decomposition of
$$\hat A^{(\lambda(t-v))}(z)=\left(
\begin{array}{cc}
\lambda(t-v(z))& -\lambda^{-1}\\
\lambda & 0
\end{array}
\right).$$ Let $r(z,t)=t-v(z)$; then $r(z,t)$ is uniformly bounded
on $\Omega_{\delta}\times[-2,2]$. If we set
$$a=a(z,t,\lambda)=|r|^2+1+\frac{1}{\lambda^4}+
\sqrt{(|r|^2+1+\frac{1}{\lambda^4})^2-\frac{4}{\lambda^4}},$$ then
obviously $a$ and $a^{-1}$ are uniformly bounded for all
$(z,t,\lambda)\in\Omega_{\delta}\times[-2,2]\times[\lambda_0,
\infty)$, where $\lambda_0$ is any large positive number. Then a
direct computation shows that $\|\hat
A\|=\lambda\sqrt{\frac{a}{2}}$. Thus $\|\hat A\|$ is uniformly of
size $\lambda$ as $\lambda\rightarrow\infty$. We also have
$$\hat U_2=\frac{1}{\sqrt{(a-\frac{2}{\lambda^4})^2+\frac{4}{\lambda^4}|r(z)|^2}}
\begin{pmatrix}
a-\frac{2}{\lambda^4}& \frac{2}{\lambda^2}\overline{r(z)}\\
-\frac{2}{\lambda^2}r(z)& a-\frac{2}{\lambda^4}
\end{pmatrix}.$$ For simplicity let
$f(z,t,\lambda)=1/\sqrt{(a-\frac{2}{\lambda^4})^2+\frac{4}{\lambda^4}|r(z)|^2}$.
Then we get the corresponding upper-left element of $\hat U$ is
\begin{center}
$\hat c(z;\alpha,
t;\lambda)=c_4\left\{r(z-\alpha)-\frac{2\overline{r(z)}}{\lambda^2a(z)}
+\frac{2r(z-\alpha)}{\lambda^4a(z)}-\frac{4\overline{r(z)}}{\lambda^6
a(z-\alpha)a(z)}\right\},$
\end{center}
where
$$c_4=\sqrt{\frac{2}{a(z-\alpha)}}f(z-\alpha)f(z)a(z-\alpha)a(z)$$
satisfies that $c_4$ and $c_{4}^{-1}$ are uniformly bounded for all
$(z,\alpha,t,\lambda)\in\Omega_{\delta}\times\R/\Z\times[-2,2]\times[\lambda_0,
\infty)$.\\

\begin{raggedright}\textit{Step} 2.\end{raggedright} Let $g(z;\alpha,t;\lambda)=\frac{1}{c_4}\hat
c(z;\alpha,t;\lambda)$; then we can reduce the analysis of uniform
positivity of $|\hat c(z;\alpha,t;\lambda)|$ to that of
$|g(z;\alpha,t;\lambda)|$. Note that
$$g(z;\alpha,t;\infty)=t-v(z-\alpha).$$

\begin{raggedright}\textit{Step} 3.\end{raggedright}
Now by analyticity and non-constancy of $v$, for each $t\in[-2,2]$,
we can pick some height $y_t$ such that
$$\mid t- v(x+iy_t-\alpha)\mid$$ is bounded away from
zero for all $(x,\alpha)\in\R/\Z\times\R/\Z$. Then for fixed $y_t$,
there is a small open interval $I_t$ around $t$ and a large
$\lambda_t>0$ such that
$$|g(x+iy_t;\alpha,s;\lambda)|$$
is bounded away from zero for all
$(x,\alpha,s,\lambda)\in\R/\Z\times\R/\Z\times I_t\times[\lambda_t,
\infty)$. Here we use the obvious fact that for fixed $t$ and $y_t$,
as $(s,\lambda)\rightarrow(t,\infty)$,
$g(x+iy_t;\alpha,s;\lambda)\rightarrow g(x+iy_t;\alpha,t;\infty)$ in
$C^0(\R/\Z\times\R/\Z,\R)$ as function of $(x,\alpha)$.

Now by compactness of $[-2,2]$ we can find finitely many $t$, say
$t_1,\cdots,t_l$, such that
\begin{enumerate}
\item $[-2, 2]\subset\bigcup_j I_{t_j}$, and
\item $|g(x+iy_{t_j};\alpha,s;\lambda)|$
bounded away from zero uniformly for all\\
$(x,s,\lambda,\alpha)\in\R/\Z\times\R/\Z\times I_t\times[\lambda_t,
\infty)$.
\end{enumerate}
This implies $|\hat c(x+iy_j;\alpha,s;\lambda)|$ are uniformly
bounded away from zero for all
$(x,\alpha,s,\lambda)\in\R/\Z\times\R/\Z\times I_t\times[\lambda_t,
\infty)$, and $\|\hat
A(x+iy_{t_j})\|$ is uniformly of size $O(\lambda)$.\\

\begin{raggedright}\textit{Step} 4-\textit{Step} 6\end{raggedright} are the same with that of section
4.1. This addresses the case $(t,\lambda)\in[-2,
2]\times[\lambda_0,\infty)$.\\

On the other hand, condition (2) above concerning the estimates of
$|g(x+iy_{t_j};\alpha,t;\lambda)|$ is automatically satisfied for
all $(t,\lambda)\in(\R\setminus[-2, 2])\times[\lambda_0,\infty)$ for
large $\lambda_0>0$, because $|v(z)|\leq 1$ on $\Omega_{\delta}$.
Since $\|\hat A(x+iy_{t_j})\|$ is uniformly of size $O(\lambda)$, we
can apply Lemma 11 to these parameters simultaneously. Which finish
the proof of Theorem A$'$.

Note all the necessary estimates in Schr\"odinger case are for all
$\alpha\in\R/\Z$, which illustrates the difference between Szeg\H o
and Schr\"odinger cases.

\section{\textbf{Nonuniform hyperbolicy}}

In this section we prove Theorem B and Theorem B$'$ in an unified
way. The main result we are going to use is the main theorem in
\cite{young}. Let's first state it and give some discussion to make
the application easier.

\subsection{\textbf{Young's theorem for nonuniformly hyperbolic $SL(2,\R)$ cocycles}}

Now let $A(\cdot,t)\in C^r(\R/\Z, SL(2,\R))$, $t\in [0,1]$ and
$r\geq 1$, be a one parameter family of cocycle maps which is $C^1$
in $(x, t)$; let $(\alpha, \Lambda(\cdot,t)U(\cdot,t))$ be the
corresponding equivalent systems and $c(x,t)$ is the upper left
element of $U(x,t)$; let

$$B(x,t,\lambda)=\begin{pmatrix}\lambda^{-1}& 0\\0& \lambda\end{pmatrix}\Lambda(x,t)
=\begin{pmatrix}b(x,t,\lambda)&
0\\0&b(x,t,\lambda)^{-1}\end{pmatrix}$$ and
$$\Delta_{\epsilon}(\lambda,\alpha)=\{t:(\alpha, A_t)\in \CN\CU\CH\ with\
L(\alpha, A_t)>(1-\epsilon)\ln\lambda\}.$$ Then the following
theorem is in \cite{young} (see \cite{young}, Theorem 2)

\begin{theorem}
Fixing arbitrary $\epsilon>0$. Let $\alpha$ be a Brjuno number and
$A(\cdot, t)$ as above; assuming $|b(x,t,\lambda)^{\pm1}|,
|\frac{\partial b(x,t,\lambda)}{\partial x}|$ are uniformly bounded
for all $(x,t,\lambda)\in\R/\Z\times[0,1]\times[\lambda_0, \infty)$
for some $\lambda_0$ large. If $c(x,t)$ is such that for each $t$
\begin{enumerate}
\item $\{x:c(x,t)=0\}\neq\varnothing$ and is finite,
\item $x\mapsto c(x,t)$ is transversal to $\{x\equiv 0\}$,
\item $\frac{\partial c}{\partial t}/\frac{\partial c}{\partial x}$
takes different values at different zeros of $c(x,t)$.
\end{enumerate}
Then $Leb(\Delta_{\epsilon}(\lambda,\alpha))\rightarrow 1$ as
$\lambda\rightarrow\infty$.
\end{theorem}

Conditions (1)-(3) are not exactly these in \cite{young}, but it's
not difficult to see the equivalence. Indeed, in \cite{young} the
author identifies $\R\PP^1=\R\cup\{\infty\}$ with $\R/(\pi\Z)$,
where the M\"obius transformation of $BU$ on $\R\PP^1$ is conjuagted
to be $\overline{BU}:\R/(\pi\Z)\rightarrow\R/(\pi\Z)$ via the
commutative diagram
\[
\xymatrix{
\R/(\pi\Z)\ar[d]_{\cot}\ar[r]^{BU}&\R/(\pi\Z)\ar[d]^{\cot}\\
\R\cup\{\infty\}\ar[r]^{BU\cdot}&\R\cup\{\infty\}}
\]
Then she considers the function
$\beta(x,t)=(\overline{(BU)}(x,t))^{-1}(\frac{\pi}{2})$; then
condition (1)-(3) is stated in terms of $\beta(x,t)$ in her main
theorem. If we let $c(x,t)=\cos(\varphi(x,t))$, then the relation
between $\beta(x,t)$ and $c(x,t)$ is
$$\beta(x,t)=\frac{\pi}{2}-\varphi(x,t)+m\pi.$$
Then it's easy to see the equivalence of these conditions in
$c(x,t)$ and in $\beta(x,t)$. In fact, the equivalent conditions can
be stated using any of the following functions:
\begin{center}
$\beta(x,t)$, $\tan(\beta(x,t))=\cot(\varphi(x,t))$,
$c(x,t)=\cos(\varphi(x,t))$ and $\varphi(x,t)-\frac{\pi}{2}+m\pi$.
\end{center}
Furthermore, $c(x,t)=\cos(\varphi(x,t))$ can also be replaced by
$f(x,t)c(x,t)$, where $f(x,t)$ satisfying that
$|f(x,t,\lambda)^{\pm1}|, |\frac{\partial f(x,t,\lambda)}{\partial
x}|$ is uniformly bounded for all
$(x,t,\lambda)\in\R/\Z\times[0,1]\times[\lambda_0, \infty)$.

Note condition (2) in Theorem 12 doesn't imply the finiteness of
$\{x:c(x,t)=0\}$ since we have functions like
$x^3\cos(\frac{1}{x})$. There is no particular reason we should use
$t\in [0,1]$. In fact, $[0,1]$ can be replaced by any connected
interval and the result still holds (this also implies conditions
(1)-(3) can be violated at finitely many $t$).

\subsection{\textbf{Proof of Theorem B}}

What we need to consider $\hat{A}=Q^*A^{(t,\lambda v)}Q\in
SL(2,\R)$. By the definition of $\hat A$ and the fact that
$Q\in\U(2)$, we have $\|\hat
A(x)\|=\sqrt{\frac{1+\lambda}{1-\lambda}}$ (this implies the
corresponding $b(x,t)$ in Theorem 12 is constantly 1, which
obviously satisfies the conditions) and the corresponding $\hat
U_2(x)=Q^*U_2(x)P(x)\in SO(2)$, where
$$P=\begin{pmatrix}
\frac{e^{\frac{\pi i}{4}}}{\sqrt{Ev(x)}}& 0\\
0& \sqrt{Ev(x)}e^{-\frac{\pi i}{4}}
\end{pmatrix}.$$ Let $(\alpha,\Lambda\hat U)$ be the corresponding equivalent
system with $\hat U\in SO(2)$. Then a direct computation shows that
the corresponding $\hat c(x,t)$ of $\hat U(x)$ is

$$\hat c(x;\alpha,t)=\cos\pi[\theta(x)-\theta(x-\alpha)+k\alpha+t].$$
Hence by the discussion following Theorem 12, we can instead
consider
$$\varphi(x,t,\alpha)=\theta(x)-\theta(x-\alpha)+k\alpha+t+m-\frac{1}{2}$$

Let $f(x)=\theta(x-\alpha)-\theta(x)-\frac{1}{2}-k\alpha$. We set
for each Brjuno $\alpha$, $I_{\alpha}=f(\R/\Z)\cap\R/\Z$. Then we
can show $\varphi(x,t,\alpha)$ satisfies conditions (1)-(3) in
Theorem 12 for each $t\in I_\alpha$. Indeed, (1) is obviously
satisfied for each $t\in I_{\alpha}$ by our choice of $\theta(x)$.
For (2), we note for each irrational $\alpha$, $\frac{\partial
\varphi}{\partial x}(x,t,\alpha)=\theta'(x)-\theta'(x-\alpha)=0$ has
only two solutions which are independent of $t$. Thus, except
finitely many $t$, (2) is satisfied for each $t\in I_{\alpha}$. For
(3), we note $\frac{\partial \varphi}{\partial t}\equiv 1$ is
nonzero and independent of $x$; thus we only need that
$\frac{\partial \varphi}{\partial x}$ takes different values for
different zeros of $\varphi(x,t)$, which again is equivalent to that
$\theta'(x)-\theta'(x-\alpha)$ takes on distinct values. But we know
for irrational $\alpha$, $\theta'(x)-\theta'(x-\alpha)=0$ exactly at
two points; furthermore we also assumed
$Leb[\theta(\R/\Z)]\leq\frac{1}{2}$, which implies
$Leb[(\theta(\cdot)-\theta(\cdot-\alpha)](\R/\Z))\leq 1$. These
together imply that $\theta'(x)-\theta'(x-\alpha)$ takes different
values at level sets of $\theta(x)-\theta(x-\alpha)$.

Now set $\CK=\overline{\cup_{\alpha}I_{\alpha}}$ (which is obviously
a compact region in $\R/\Z\times\R/\Z$) and let
$s(\alpha,\lambda)=Leb(I_{\alpha}\cap\Delta_{\epsilon}(\lambda,\alpha))$;
then Theorem 12 implies for $Lebsgue$ almost every $\alpha$,
$s(\alpha,\lambda)\rightarrow Leb(I_{\alpha})$ as
$\lambda\rightarrow 1$. Hence, Bounded Convergence Theorem implies
$$Leb(\CK\cap\Gamma_{\epsilon}(\lambda))=\int_{\R/\Z}s(\alpha,\lambda)d\alpha\rightarrow
\int_{\R/\Z}Leb(I_{\alpha})d\alpha=Leb(\CK)$$ as $\lambda\rightarrow
1$. This completes the proof.

\begin{remark}
By the proof of the main theorem, it's easy to see that any region
in $\R/\Z\times\R/\Z$ which is away from $\CK$ with a positive
distance is in $\CU\CH$ as $\lambda\rightarrow 1$. On the other
hand, obviously we have $Leb(I_{\alpha})\rightarrow 0$ as
$\alpha\rightarrow 0$. If we take $v(x)=e^{\pi i\cos(2\pi x)}$, then
$Leb(I_{\alpha})\rightarrow 1$ as $\alpha\rightarrow\frac{1}{2}$.
Thus $Leb(\Sigma)\rightarrow 1$ as $\lambda\rightarrow 1$ and Brjuno
$\alpha\rightarrow\frac{1}{2}$. Namely, we've constructed some
analytic quasiperiodic 2-sided Verblunsky coefficients, of which the
associated $\mu_x$ satisfying that $supp(\mu_x)$ can be arbitrarily
close to full measure.
\end{remark}

\subsection{\textbf{Proof of Theorem B$'$}}

For the proof of Theorem B$'$ we continue to use the polar
decomposition of proof of Theorem A$'$ in Section 4.2. Note here we
assume $v(\R/\Z)=[0,1]$ and we restrict to $y=0$.

Recall we have $\|\hat A\|=\lambda\sqrt{\frac{a}{2}}$, where for
$t=\frac{E}{\lambda}$ and $r(x)=t-v(x)$,
$$a=a(x,t,\lambda)=r(x)^2+1+\frac{1}{\lambda^4}+
\sqrt{(r(x)^2+1+\frac{1}{\lambda^4})^2-\frac{4}{\lambda^4}}$$
satisfies that $a$ and $a^{-1}$ are uniformly bounded for all
$(x,t,\lambda)\in\R/\Z\times[0,1]\times[\lambda_0, \infty]$. Thus
the corresponding $\hat b(x,t,\lambda)=\sqrt{\frac{a}{2}}$ satisfies
all the conditions in Theorem 12. We also have
$$\hat
c(x;\alpha,t;\lambda)=c_4\left\{r(x-\alpha)-\frac{2r(x)}{\lambda^2a(x)}
+\frac{2r(x-\alpha)}{\lambda^4a(x)}-\frac{4r(x)}{\lambda^6
a(x-\alpha)a(x)}\right\},$$ where
$$c_4=\sqrt{\frac{2}{a(x-\alpha)}}f(x-\alpha)f(x)a(x-\alpha)a(x)$$
and
$f(x,t,\lambda)=1/\sqrt{(a-\frac{2}{\lambda^4})^2+\frac{4}{\lambda^4}r(x)^2}$.
Hence, we have $$\hat
c(x;\alpha,t,\infty)=\frac{t-v(x-\alpha)}{\sqrt{(t-v(x-\alpha))^2+1}}.$$
Furthermore it's not difficult to see that for any fixed $\alpha$,
\begin{center}
$\hat c(x;\alpha,t;\lambda)\rightarrow\hat c(x;\alpha,t;\infty)$ in
$C^1(\R/\Z\times[0,1],\R)$ as $\lambda\rightarrow\infty$.
\end{center}
Indeed, it's easy to see this reduces to the convergence of
$a(x,t,\lambda)$ to $a(x,t,\infty)$ in $C^1$ topology, which is
immediate.

Now by the discussion following Theorem 12, conditions (1)-(3) for
$\hat c(x,t)$ in Theorem 12 can be reduced to these for
$t-v(x-\alpha)$ for sufficiently large $\lambda$, which are
immediate by our assumption on $v(x)$ in Theorem B$'$. Indeed,
outside of a finite set of $t$ containing critical values of $v$,
assumption (1) implies condition (1)-(2) of Theorem 15 and
assumption (2) is the same with condition (3).

By exactly the same proof of Theorem B, it follows that

\begin{center}
$Leb((\R/\Z\times[0,1])\cap\Gamma_{\epsilon}(\lambda))\rightarrow 1$
as $\lambda\rightarrow \infty$.
\end{center}
This completes the proof of Theorem B$'$.

\begin{remark}
In fact, the proof of Theorem 12 in \cite{young} contains more
information: for each $\alpha$ Brjuno and
$t\in\Delta_{\epsilon}(\lambda,\alpha)$, there are finitely many
$x_i\in\R/\Z$, say $i=1,\cdots, l$, along whose orbits there are
sequences $\{\|A_{n}^{(E,\lambda v)}(x_i)w\|\}_{n\in\Z}$ (for some
$w\in\R^2\setminus\{0\}$) which decay exponentially as
$n\rightarrow\pm\infty$. Namely, let
$E\in\lambda\Delta_{\epsilon}(\lambda,\alpha)$ and $x_i\in\R/\Z$ be
critical points of $(\alpha,A^{(E-\lambda v)})$; then $E$ is an
eigenvalue of operator $H_{\alpha,\lambda,x_i}$ with exponentially
decay eigenfunction. This is the so-called Anderson Localization.
$\{x_1, \cdots, x_l\}$ are the so-called `critical set', the
existence of which leads to nonuniform hyperbolicity. Another fact
that we can get from the proof of Theorem 12 is that in our
Schr\"odinger cocycle setting, both `initial critical set' (zeros of
$\hat c(x;\alpha,t)$) and `critical set' converges to zeros of
$t-v(x-\alpha)$ as $\lambda\rightarrow\infty$. This gives an
explicit description of number of `critical points'. We can relax
our conditions on $v(x)$ to get weak results. Namely any $C^1$
potential function $v$ with an interval $I$ in its image satisfying
condition (1)-(3) can be our choice (for example, we can relax
$v(x)$ to have an unique minimum or maximum and pick a small
interval $I$ of $t$ near this minimal or maximal value). This method
obviously allows us to produce cocycles $(\alpha,A^{E-\lambda v})$
to have $2n$ `critical points', where $n\geq1$ can be any natural
number.

There is another theorem in \cite{young} (see \cite{young}, Theorem
1) taking frequency $\alpha$ as parameter, the proof of which is
basically the same with that of Theorem 12. Combining this theorem
with sard theorem and our decomposition procedure, it's easy to see
that if we fix arbitrarily $C^1$ potential $v$ and arbitrary
$\epsilon>0$, then for almost every $E\in \lambda v(\R/\Z)$,
$$\lim\limits_{\lambda\rightarrow\infty}Leb\{\alpha: (\alpha,A^{(E-\lambda v)})
\in \CN\CU\CH\ \mbox{and}\ L(\alpha,A^{(E-\lambda
v)})>(1-\epsilon)\ln\lambda\}=1.\footnote{The author is grateful to
Anton Gorodetski and Vadim Kaloshin for showing me their notes,
where they pointed this out.}$$
\end{remark}

\section{\textbf{Discussion}}
In Schr\"odinger case, under the condition $|v|\leq1$ on
$\Omega_{\delta}$, a little bit more computation shows that for any
fixed $\lambda>1$, $(\alpha,A^{(E-\lambda v)})$ satisfies all the
conditions in Lemma 11 for all $E\in\R\setminus[-3-\lambda,
3+\lambda]$ (In the case of Theorem B$'$, this also implies that
$\{E\in\R:L(E)>(1-\epsilon)\ln\lambda\}$ tends to be full measure
set as $\lambda\rightarrow\infty$). The \textit{uniform
hyperbolicity} result implied by Lemma 11 is nothing new. Because
$\R\setminus[-2-\lambda,2+\lambda]$ is in the resolvent set and we
have the basic fact related \textit{uniform hyperbolicity} and
resolvent set. The new fact is that by Lemma 11, for $E$ and
$\lambda$ in the case above, $(\alpha, A^{(E-\lambda v)})$ admits an
invariant cone field such that for each vector in this cone, it's
expanded under forward iteration on each step. This is
\textit{uniform hyperbolicity} in some strong sense and is not true
for $E$ in the spectral gap.

More concretely, as in the proof of Theorem 9, $(\alpha,
A^{(E-\lambda v)})\in\CU\CH$ can be analytically conjugated to a
diagonal system
$$\left(\alpha, \left(
\begin{array}{cc}
r(x)& 0\\
0&r(x)^{-1}
\end{array}
\right)\right)$$ via its stable and unstable direction (where we
assume $r(x)$ corresponds to unstable direction). Then for
$(E,\lambda)$ satisfying conditions in Lemma 11, we have $|r(x)|\geq
c>1$ for all $x\in\R/\Z$. We emphasize here that the invariant cone
field are for the original system $(\alpha, A^{(E-\lambda v)})$.
Because for $\alpha$ and $v\in C_{\delta}^{\omega}(\R/\Z,\R)$
satisfying
$$\delta>\beta(\alpha)=\limsup\limits_{n\rightarrow\infty}\frac{\ln q_{n+1}}{q_n}$$
(where $\frac{p_n}{q_n}$ is the continued fraction approximants of
$\alpha$), we can always analytically conjugate the diagonal system
to the constant system

$$\left(\alpha, \left(
\begin{array}{cc}
e^{L(\alpha,A^{(E-\lambda v)})}& 0\\
0&e^{-L(\alpha,A^{(E-\lambda v)})}
\end{array}
\right)\right)$$ by solving a simple cohomological equation: $$\ln
|b(x+\alpha)|-\ln|b(x)|=\ln|r(x)|-L(\alpha,A^{(E-\lambda v)}).$$

Note that in both cases, for large couplings, one can always choose
suitable height $y$ such that $(\alpha, A_y)$ is \textit{uniformly
hyperbolic}. This fact reflects another important theorem in
\cite{avila2}. Recall that for irrational $\alpha$, $y\mapsto
L(\alpha,A_y)$ is a piecewise affine function in $y$, thus it
natural to give the following definition

\begin{definition}
We say that $(\alpha,A)\in ((\R/\Z)\setminus\Q)\times
C^{\omega}(\R/\Z,SL(2,\C))$ is $regular$ if $L(\alpha,A_y)$ is
affine for $y$ in a neighborhood of 0.
\end{definition}

Then the theorem in \cite{avila2} (see \cite{avila2}, Theorem 6) is
\begin{theorem}
Assume that $L(\alpha,A)>0$. Then $(\alpha, A)$ is regular if and
only if $(\alpha, A)$ is uniformly hyperbolic.
\end{theorem}

Thus if a priori we know $L(\alpha,A)>0$, then for each nonzero
height $y$ sufficiently close to $0$, $(\alpha,A_y)$ is always
\textit{uniformly hyperbolic}.

In both cases, the essential obstruction is that $c(x;\alpha,t)$
oscillates around $0$, which forces us to consider the real analytic
case and choose suitable heights to get uniformly positive Lyapunov
exponents.

On the other hand, we know for some $t$ and
$\{x:c(x;\alpha,t)=0\}\neq\varnothing$, we can still have uniformly
hyperbolic systems. For example, these $E$ in the spectral gap of
Almost Mathieu operators ($v(x)=2\cos(2\pi x)$) with $\lambda>2$ can
be our choice (see \cite{avilajitomirskaya}, Main Theorem). Thus it
will be very interesting to understand that, under the condition
$\{x:c(x,t)=0\}\neq\varnothing$ and uniformly positive Lyapunov
exponents, how can one go between {\it uniform hyperbolicity} and
{\it nonuniform hyperbolicity} when $t$ varies.

The proof of Theorem 12 implies that for in the case of Theorem
B$'$, it's exactly these $t$, near which resonance occurs or near
critical value of $v$, that have been excluded. Here, for example,
since the induction step gets started at the continued fractional
approximant $q_N$ for some large $N$, resonance at initial step
means that there exist some $x_0$ and $1\leq k<q_N$ such that
$$t=v(x_0)\mbox{ and }|v(x_0)-v(x_0+k\alpha)|<<\frac{1}{q_N^2}.$$ Thus the
natural thing to do next is to study these $t$ and do the following
possible generalization: for a fixed {\em Diophantine} frequency,
put some additional conditions on $v$, like higher but finite
regularity (for example, $C^3$) and nondegeneracy of critical points
(i.e. $v''(x_0)\neq0$ where $v'(x_0)=0$) to get the positive
Lyapunov exponents for all $E$ for sufficient large couplings (the
difference between this result and that in \cite{chan} is the
following: here one try to fix frequency and potential while in
\cite{chan} the author eliminates frequencies and varies
potentials). As we stated in Section 1.2.3, a new induction step is
needed to take care of appearance and disappearance of `critical
points' near resonance. One can even try to prove Anderson
Localization (AL) for almost every phase or try to produce
counterexamples such that AL does not hold. Similar problems are
also proposed in \cite{klein}.

\end{document}